\documentclass{siamltex}
\usepackage{amssymb,amsmath,amsfonts}
\usepackage[ruled,section]{algorithm}
\usepackage{graphics}
\usepackage{algorithmic,verbatim,url,comment,color}

\newcommand{\klein}[1]{{\footnotesize{#1}}}

\newcommand{\E}{\mathbf{E}}
\newcommand{\Prob}{\mathbf{Pr}}

\newtheorem{remark}[theorem]{Remark}

\newenvironment{preview}[1][Preview]{\medskip\noindent
\begin{trivlist}
\item{\hskip \labelsep{\scshape Preview of #1\ }}
\it\ignorespaces}{\end{trivlist}}

\title{The effect of coherence on sampling from matrices with orthonormal 
columns, and preconditioned least squares problems\thanks{Both authors were
supported in part by NSF grant CCF-1145383. The first author also
acknowledges the support from the XDATA Program of the Defense Advanced
Research Projects Agency (DARPA), administered through Air Force
Research Laboratory contract FA8750-12-C-0323 FA8750-12-C-0323.
}
}
\author{Ilse C. F. Ipsen\thanks{%
Department of Mathematics, North Carolina State University, P.O. Box 8205,
Raleigh, NC 27695-8205, USA, (\texttt{ipsen@ncsu.edu}, 
\texttt{http://www4.ncsu.edu/{\char'176}ipsen/})}
\and
Thomas Wentworth\thanks{%
Department of Mathematics, North Carolina State University, P.O. Box 8205,
Raleigh, NC 27695-8205, USA (\texttt{thomas\_wentworth@ncsu.edu})}
}
\begin{document}
\maketitle

\begin{abstract} 
Motivated by the least squares solver \textsl{Blendenpik}, we
investigate three strategies for uniform sampling of rows from
$m\times n$ matrices $Q$ with orthonormal columns. The goal is to
determine, with high probability, how many rows are required so that
the sampled matrices have full rank and are well-conditioned with
respect to inversion.

Extensive numerical experiments illustrate that the three sampling
strategies (without replacement, with replacement, and Bernoulli
sampling) behave almost identically, 
for small to moderate amounts of sampling.
In particular, sampled matrices of full rank tend to have
two-norm condition numbers of at most~10.

We derive a bound on the condition number of the sampled matrices
in terms of the coherence $\mu$ of $Q$. This bound
applies to all three different sampling strategies; it
implies a, not necessarily tight, lower bound of
$\mathcal{O}\left(m\mu\ln{n}\right)$ for the number of sampled rows;
and it is realistic and informative 
even for matrices of small dimension
and the stringent requirement of a 99 percent success probability.

For uniform sampling with replacement we derive a potentially tighter
condition number bound in terms of the leverage scores of $Q$. To obtain
a more easily computable version of this bound, in terms of just the largest
leverage scores, we first derive a general bound on the two-norm of 
diagonally scaled matrices.

To facilitate the numerical experiments and test the tightness of the bounds,
we present algorithms to generate matrices with user-specified 
coherence and leverage scores. These algorithms, the three sampling strategies,
and a large variety of condition number bounds are implemented
in the Matlab toolbox \textsl{kappa\_SQ\_v3}.
\end{abstract}


\begin{AM} 
65F08, 65F10, 65F20, 65F25, 65F35, 68W20, 15A12, 15A18, 15A42, 15B10, 15B52
\end{AM}

\section{Introduction}
Our paper was inspired by Avron, Maymounkov and Toledo's
\textsl{Blendenpik} algorithm and analysis \cite{AMTol10}.

\textsl{Blendenpik} is an iterative method for solving overdetermined
least squares/regres\-sion problems $\min_x{\|Ax-b\|_2}$ with the Krylov
space method \textsl{LSQR} \cite{PS82}.  In order to accelerate
convergence, \textsl{Blendenpik} constructs a preconditioner $R_s$ and solves
instead the preconditioned least squares problem
$\min_z{\|AR_s^{-1}z-b\|_2}$. The solution to the original problem 
is recovered by solving a linear system with coefficient matrix~$R_s$.
The innovative feature is the construction of the preconditioner $R_s$
by a random sampling method.

\subsection{Motivation}
The purpose of our paper is a thorough experimental and analytical
investigation of random sampling strategies for producing efficient
preconditioners.  The challenge is to ensure not only that $R_s$ is
nonsingular, but also that $AR_s^{-1}$ is well-conditioned with
respect to inversion, which is required for fast convergence and numerical 
stability.

Here is a conceptual point of view of how \textsl{Blendenpik}
constructs the preconditioner: First it ``smoothes out'' the rows of $A$ 
by applying a randomized unitary transform $F$, and then it uniformly samples
(i.e. selects) a small number of rows $M_s$ from $FA$. At last it
computes a QR factorization of the smaller sampled matrix,
$M_s=Q_sR_s$, where the triangular factor $R_s$ serves as the preconditioner.

The neat and crucial observation in \cite{AMTol10} is to realize that sampling
rows from $FA$ amounts, conceptually, to sampling rows from an
orthonormal basis of $FA$. That is, if the columns of $Q$ 
represent an orthonormal basis for the column space of $FA$,
and if $S$ is a sampling matrix then $SQ$ has the same
two-norm condition number as $AR_s^{-1}$.  This means, it suffices to
consider sampling from matrices $Q$ with orthonormal columns.

The analysis in \cite{AMTol10} suggests that 
$SQ$ is well conditioned, if $Q$ has low ``coherence''.  Intuitively,
coherence gives information about the localization
or ``uniformity'' of the elements of $Q$.  Mathematically,
coherence is the largest (squared) norm of any row of $Q$.
For instance, if $Q$ consists of canonical vectors, then the
non-zero elements are concentrated in only a few rows, so that $Q$
has high coherence.  However, if $Q$ is a submatrix of a Hadamard
matrix, then all elements have the same magnitude, so that $Q$
has low coherence.

If $Q$ has low coherence, then, in the context of sampling, all rows
are equally important. Hence any sampled matrix $SQ$ with sufficiently
many rows is likely to have full rank. The purpose of the randomized
transform $F$ is to produce a matrix $FA$ whose orthonormal basis $Q$
has low coherence.

We were intrigued by the analysis of \textsl{Blendenpik} because 
it appears to be the first to exploit the concept of coherence for numerical 
purposes. We  also wanted to get a better understanding of 
the condition number bound for $SQ$ in
\cite[Theorem 3.2]{AMTol10}, which contains an unspecified constant,
and of the effect of uniform sampling strategies.

\subsection{Overview and main results}
We survey the contents of the paper, with a focus on the main results.

\subsubsection*{From preconditioned matrices to sampled matrices with 
orthonormal columns (Section~\ref{s_bpik})}
We start with a brief sketch of the
\textsl{Blendenpik} least squares solver (Section~\ref{s_bpikalg}), and
make the important transition from preconditioned matrices $AR_s^{-1}$
to sampled matrices $SQ$ with orthonormal columns, made possible by
the observation (\cite{AMTol10,RokT08} and Lemma~\ref{l_1}) that both 
have the same two-norm condition 
number\footnote{Here $\kappa(X)\equiv \|X\|_2\>\|X^{\dagger}\|_2$ denotes the 
Euclidean two-norm 
condition number with respect to inversion of a full rank matrix $X$.
The matrix $X^{\dagger}$ is the Moore-Penrose inverse of $X$.},
$$\kappa(AR_s^{-1})=\kappa(SQ).$$
Then we discuss the notion of coherence and its properties 
(Section~\ref{s_coh}). For a $m\times n$ matrix $Q$ with orthonormal 
columns, $Q^TQ=I_n$, the \textit{coherence} 
$$\mu\equiv\max_{1\leq j\leq m}{\|e_j^TQ\|_2^2}$$
is the largest squared row norm\footnote{The superscript $T$ denotes transpose,
and $I_n$ is the $n\times n$ identity matrix with columns $e_j$.}.

\subsubsection*{Sampling methods (Section~\ref{s_samp})}
We discuss three randomized methods for producing sampling matrices~$S$: 
Sampling without replacement (Section~\ref{s_without}), sampling
with replacement (Section~\ref{s_with}), and Bernoulli sampling
(Section~\ref{s_bernoulli}).  We show that Bernoulli sampling can be
viewed as a form of sampling without replacement (Section~\ref{s_relate}).

The sampling matrices $S$ from all three methods are constructed so that 
$S^TS$ is an unbiased estimator of the identity matrix.
The action of applying $S$ to a matrix $Q$ with orthonormal columns,
$SQ$, amounts to randomly sampling rows from~$Q$.

The numerical experiments (Section~\ref{s_sampcomp}) illustrate two points:
First, the three sampling methods behave almost identically, 
in terms of the percentage of sampled matrices $SQ$ that have full
rank and their condition numbers, in particular for small to moderate
sampling amounts. Second, those sampled matrices $SQ$ that have 
full rank tend to be
very well-conditioned, with condition numbers $\kappa(SQ)\leq  10$.

As a consequence  (Section~\ref{s_sampconc}),
we recommend sampling with replacement for  \textsl{Blendenpik}, because
it is fast, and it is easy to implement.

\subsubsection*{Numerical experiments}
Since random sampling methods can be expected to work well in the
asymptotic regime of very large matrix dimensions, we restrict all
numerical experiments to matrices of small dimension.

Furthermore, we consider only  matrices that have many
more rows than columns, $m\gg n$.  This is the situation where
random sampling methods can be most efficient.  In contrast,
random sampling methods are not efficient for matrices that are almost
square, because the number of rows in $SQ$ has to be at least
equal to $n$, otherwise $\rank(SQ)=n$ is not possible.

\subsubsection*{Condition number bounds based on coherence 
(Section~\ref{s_cohbound})}
We derive a probabilistic  bound, in terms of coherence, 
for the condition numbers of the sampled  matrices
(Theorem~\ref{t_cohbound} in Section~\ref{s_subcohbound}). The bound
applies to all three sampling methods.
From this we derive the following lower bound, not necessarily 
tight, on the required number of sampled rows.

\begin{preview}[Corollary~\ref{c_cohbound}]
Given a failure probability $0<\delta<1$, and a 
tolerance $0\leq \epsilon<1$.  To achieve the condition number bound
$\kappa(SQ)\leq \sqrt{\tfrac{1+\epsilon}{1-\epsilon}}$,
the number of rows from $Q$, sampled by any of three methods,
should be at least
\begin{eqnarray}\label{e_cohbound}
c\geq 3m\mu\>\frac{\ln(2n/\delta)}{\epsilon^2}.
\end{eqnarray}
\end{preview}

This suggests that one has to sample more rows for $SQ$
if $Q$ has high coherence ($\mu$ close to 1), 
if one wants a low condition number bound (small $\epsilon$), or
if one wants a high success probability (small $\delta$).

Numerical experiments (Section~\ref{s_cohex}) illustrate that the
bounds are informative for matrices with sufficiently low coherence $\mu$ and
sufficiently high aspect ratio $m/n$.
Our bounds have the following advantages (Section~\ref{s_cohconc}):
\begin{enumerate}
\item They are tighter than those in \cite[Theorem 3.2]{AMTol10}
because they are non-asymp\-totic, with all constants explicitly specified.
\item They apply to three different sampling methods.
\item They imply a lower bound, of $\Omega\left(m\mu\ln{n}\right)$,
on the required number of sampled rows.
\item They are realistic and informative --
even for matrices of small dimension
and the stringent requirement of a 99 percent success probability.
\end{enumerate}

\subsubsection*{Condition number bounds based on leverage scores, 
for uniform sampling with replacement (Section~\ref{s_leverage})}
The goal is to tighten
the coherence-based bounds from Section~\ref{s_cohbound} by making use of all
the row norms of $Q$, instead of just the largest one.
To this end we introduce \textit{leverage scores} 
(Section~\ref{s_levscore}), which are the squared row norms of $Q$,
$$\ell_j=\|e_j^TQ\|_2^2, \qquad  1\leq j\leq m.$$
We use them to derive a bound for uniform sampling with
replacement (Theorem~\ref{t_levbound} in Section~\ref{s_levbound}).
Then we present a more easily computable bound, in terms of just 
a few of the largest
leverage scores (Section~\ref{s_clevbound}). It 
implies the following lower bound, not necessarily
tight, on the number of samples.

\begin{preview}[Corollary~\ref{c_levbound}]
Given a failure probability $0<\delta<1$, a tolerance $0\leq \epsilon<1$,
and a labeling of leverage scores in non-increasing order,
$$\mu=\ell_{[1]}\geq \cdots\geq \ell_{[m]}.$$ 
To achieve the condition number bound 
$\kappa(SQ)\leq \sqrt{\tfrac{1+\epsilon}{1-\epsilon}}$,
the number of rows from $Q$, sampled uniformly with replacement,
should be at least
\begin{eqnarray}\label{e_lsbound}
c\geq \tfrac{2}{3}m\>(3\tau+\epsilon\mu)\>
\frac{\ln(2n/\delta)}{\epsilon^2},
\end{eqnarray}
where $t\equiv\left\lfloor 1/\mu\right\rfloor$ and
$\tau\equiv \mu\>\sum_{j=1}^t{\ell_{[j]}}+(1-t\,\mu)\,\ell_{[t+1]}$.
\end{preview}
\medskip

We show (Section~\ref{s_ancomp}) that (\ref{e_lsbound}) is
indeed tighter than (\ref{e_cohbound}).
This is confirmed by numerical experiments (Section~\ref{s_numcomp}).
The difference becomes more drastic for matrices $Q$ with
widely varying non-zero leverage scores,
and can be as high as ten percent.
Hence (Section~\ref{s_lqconc}), when it comes to lower bounds for the 
number of rows sampled
uniformly with replacement, we recommend (\ref{e_lsbound}) 
over~(\ref{e_cohbound}).

\subsubsection*{Algorithms for generating matrices with prescribed
coherence and leve\-rage scores (Section~\ref{s_genalg})} 
The purpose
is to make it easy to investigate the efficiency of the sampling
methods in Section~\ref{s_samp}, and test the tightness of the bounds
in Sections \ref{s_cohbound} and~\ref{s_leverage}.

To this end we present algorithms for generating matrices with
prescribed leverage scores and coherence (Section~\ref{s_alg}), and
for generating particular leverage score distributions with prescribed
coherence (Section~\ref{s_precoh}).  Furthermore we present two
classes of structured matrices with prescribed coherence that are easy
and fast to generate (Section~\ref{s_mat}).  The basis for the
algorithms is the following majorization result.

\begin{preview}[Theorem~\ref{t_maj}]
Given integers $m\geq n$
and a vector $\ell$ with $m$ elements that satisfy
$0\leq \ell_j \leq 1$ and $\sum_{j=1}^m{\ell_j}=n$,
there exists a $m\times n$ matrix $Q$ with orthonormal 
columns that has leverage scores $\|e_j^TQ\|_2^2=\ell_j$, $1\leq j\leq m$,
and coherence $\mu=\max_{1\leq j\leq m}{\ell_j}$.
\end{preview}

\subsubsection*{Bound for two-norms of diagonally scaled matrices
(Section~\ref{s_app2})}
The bound (\ref{e_lsbound}) is based on a special case of the 
following general
bound for the two-norm of diagonally scaled matrices. 

\begin{preview}[Theorem~\ref{t_2}]
Let $Z$ be a $m\times n$ matrix with $\rank(Z)=n$ and largest
squared row norm $\mu_z\equiv\max_{1\leq j\leq m}{\|e_j^TZ\|_2^2}$.
Let $D$ be a $m\times m$ 
non-negative diagonal matrix, and a labeling of diagonal elements
in non-increasing order,
$$\|D\|_2=d_{[1]}\geq \cdots\geq d_{[m]}\geq 0.$$ 
If $t\equiv\left\lfloor (\|Z^{\dagger}\|_2^2\>\mu_z)^{-1}\right\rfloor$, then either
$$\|DZ\|_2^2\leq \mu_z\sum_{j=1}^t{d_{[j]}^2}+
\left(\|Z\|_2^2-t\,\mu_z\right)\,d_{[t+1]}^2\qquad \text{ if } \|Z\|_2^2-t\,\mu_z\leq \mu_z$$
or
$$\|DZ\|_2^2\leq \mu_z\sum_{j=2}^{t+1}{d_{[j]}^2}+
\left(\|Z\|_2^2-t\,\mu_z\right)\,d_{[1]}^2.\qquad \text{ if } \|Z\|_2^2-t\,\mu_z > \mu_z.$$

\end{preview}

\subsubsection*{Matlab toolbox}
In order to perform the experiments in this paper, 
we developed a Matlab toolbox \texttt{kappaSQ\_v3} 
with a user-friendly interface \cite{impl}.
The toolbox contains
implementations of the three random sampling methods in Section~\ref{s_samp},
the matrix generation algorithms in Section~\ref{s_genalg}, 
the bounds in Sections \ref{s_cohbound} and~\ref{s_leverage}, and a variety
of other
condition number bounds. It also allows the user to input her/his own
matrices.

\subsubsection*{Proofs (Sections \ref{s_app}, \ref{s_app2} and~\ref{s_pre})}
All proofs, except those for Sections \ref{s_bpik} and~\ref{s_samp},
have been relegated to these three sections, which form the appendix. 

Section~\ref{s_app} contains the proofs for 
Sections \ref{s_cohbound} and~\ref{s_leverage}, which are based on 
two matrix concentration inequalities:
A Chernoff bound  (Section~\ref{s_chernoff}), and a Bernstein bound
(Section~\ref{s_bernstein}).

Section~\ref{s_app2} contains the proofs
for the easily computable bounds in Sections \ref{s_clevbound} 
and~\ref{s_ancomp},
together with the majorization results (Section~\ref{s_tns}) 
required for the proofs.

The majorization results in Section~\ref{s_pre} represent the foundation
for the algorithms in Section~\ref{s_genalg}.

\subsubsection*{Future work (Section~\ref{s_future})}
We list a few issues that suggest
themselves immediately as a follow up to this paper.

\subsection{Literature}
Existing randomized least squares methods are based on 
randomized projections. This means, conceptually they
multiply $A$ by a random matrix~$F$, and then sample
a few rows from $FA$.

The algorithms in \cite{BDr09,DMM06a,DMMS10} solve a smaller sampled
problem by a direct method.  Like \textsl{Blendenpik} \cite{AMTol10},
the algorithm
in \cite{RokT08} computes a preconditioner from the QR factorization
of a sampled submatrix, but then solves the preconditioned problem
by applying the conjugate gradient method to the normal equations.
The parallel solver \textsl{LSRN} \cite{MSM11} computes a
preconditioner from the SVD of a sampled submatrix, and then
solves the preconditioned problem with an iterative
method. This solver applies to general matrices
rather than just those of full column rank.

As for randomized algorithms in general, the excellent surveys
\cite{HMT09,Mah11} provide clear analyses and good intuition.

\subsection{Notation}
The norm $\|\cdot\|_2$ denotes the Euclidean two-norm, 
and the two-norm \textit{condition number} with respect to inversion
of a real $m\times n$ matrix $Z$ with $\rank(Z)=n$ is denoted by
$\kappa(Z)\equiv\|Z\|_2\|Z^{\dagger}\|_2$, where $Z^{\dagger}$ is the
Moore-Penrose inverse.
The $k\times k$ \textit{identity matrix} is 
$I_k=\begin{pmatrix}e_1 &\ldots &e_k\end{pmatrix}$, and its columns are
the \textit{canonical vectors} $e_j$, $1\leq j\leq k$.

The \textit{probability} of an event $\mathcal{X}$ is denoted by
$\Prob[\mathcal{X}]$, and the \textit{expected value} of a random
variable $X$ is denoted by $\E[X]$.
\section{The Blendenpik algorithm, and coherence}\label{s_bpik}
We describe the \textsl{Blendenpik} algorithm 
for solving least squares problems (Section~\ref{s_bpikalg}), and
present the notion of coherence (Section~\ref{s_coh}).
 
\subsection{Algorithm}\label{s_bpikalg}
The \textsl{Blendenpik} algorithm \cite[Algorithm 1]{AMTol10} solves
full column rank least squares problems with the Krylov space
method \textsl{LSQR} \cite{PS82} and a randomized preconditioner.
Algorithm~\ref{alg_bpik} presents a conceptual sketch of
\textsl{Blendenpik}. The subscript ``$s$'' denotes quantities
associated with the sampled matrix.

\begin{algorithm}
\caption{Sketch of \textsl{Blendenpik} \cite{AMTol10}}\label{alg_bpik}
\begin{algorithmic}
\REQUIRE $m\times n$ matrix $A$ with $m\geq n$ and $\rank(A)=n$,
$m\times 1$ vector $b$\\
$\qquad$ $m\times m$ random unitary matrix $F$\\
$\qquad$ $k\times n$ sampling matrix $S$ with $k\geq n$\\
\ENSURE Solution of $\min_x{\|Ax-b\|_2}$\\
$\qquad$\\
\STATE $M=FA$ $\qquad$ \COMMENT{Improve coherence}\\
\STATE $M_s=SM$  $\qquad$ \COMMENT{Sample for preconditioner}\\

\STATE Thin QR factorization $M_s=Q_sR_s$
$\qquad$ \COMMENT{Generate preconditioner}\\
\STATE Determine solution $y$ to $\min_z{\|AR_s^{-1}z-b\|_2}$
$\qquad$ \COMMENT{Solve preconditioned problem}\\
\STATE Solve $R_s\hat{x}=y$ $\qquad$ 
\COMMENT{Recover solution to original problem}
\end{algorithmic}
\end{algorithm}

The matrix $F$
is the product of a random diagonal matrix with $\pm 1$ entries, and
a unitary transform, such as a Walsh Hadamard transform, or a discrete Fourier,
Hartley or cosine transform \cite[Section 3.2]{AMTol10}. 
The transformed matrix $M=FA$ is $m\times n$ with $m\geq n$ and $\rank(M)=n$. 

The sampling matrix $S$ selects $k\geq n$ rows
from the transformed matrix $M$.  We discuss different types of
sampling matrices in Section~\ref{s_samp}.  The $k\times n$ sampled
matrix $M_s$ has a thin QR decomposition $M_s=Q_sR_s$ where $Q_s$ is
$k\times n$ with orthonormal columns and $R_s$ is $n\times n$ upper
triangular.

The basis for the analysis is the thin QR decomposition
$M=QR$, where $Q$ is $m\times n$ with 
orthonormal columns and $R$ is $n\times n$ upper triangular. 
This QR decomposition is \textit{not} computed.
The next result links the condition number of the 
preconditioned matrix to that of the matrix $SQ$,
see also \cite[Section~3.1]{AMTol10} and
\cite[Theorem 1]{RokT08}.

\begin{lemma}\label{l_1}
With the notation in Algorithm~\ref{alg_bpik},
if $\rank(M_s)=n$, then 
$$\kappa(AR_s^{-1})=\kappa(SQ).$$
\end{lemma}

\begin{proof}
From $FA=M=QR$  and the fact that the 2-norm is invariant
under premultiplication by matrices with orthonormal columns, it follows that
\begin{eqnarray*}
\kappa(AR_s^{-1})&=&\kappa(MR_s^{-1})=\kappa(RR_s^{-1})
=\kappa(R_sR^{-1})=\kappa(M_sR^{-1})=\kappa(SMR^{-1})\\
&=&\kappa(SQ).
\end{eqnarray*}
\end{proof}

In Sections \ref{s_cohbound} and~\ref{s_leverage}
we derive bounds for the condition number
of the preconditioned matrix, $\kappa(AR_s^{-1})$.
Our bounds are tighter than those in \cite[Theorem 3.2]{AMTol10},
because they have all constants explicitly specified,
and apply to three different sampling strategies.
Since Lemma~\ref{l_1} implies $\kappa(AR_s^{-1})=\kappa(SQ)$,
we state the bounds for $\kappa(SQ)$ only.
An important ingredient in these bounds is the coherence of~$Q$.

\subsection{Coherence}\label{s_coh}
Coherence gives information about the localization or ``uniformity''
of the elements in an orthonormal basis.  The more general concept of
\textit{mutual coherence} between two orthonormal bases was introduced
in \cite[\S VII]{DoH01}, in the context of signal processing and
computational harmonic analysis, to describe a condition for the
existence of sparse representations of signals.  What we use here is a
special case, and can be viewed as a measure for how close an
orthonormal basis is to sharing a vector with a canonical basis.

\begin{definition}[Definition 3.1 in \cite{AMTol10}, Definition 1.2
in \cite{CanR09}]\label{d_co}
Let $Q$ be  a real $m\times n$ matrix with orthonormal columns, $Q^TQ=I_n$,
then the {\rm coherence} of $Q$ is 
\begin{eqnarray*}\label{e_coherence}
\mu\equiv \max_{1\leq j\leq m}\|e_j^TQ\|_2^2.
\end{eqnarray*}

If the columns of $Q$ are an orthonormal basis for the column space of
a matrix~$M$, then the coherence of $M$ is $\mu$.
\end{definition}

The second part of Definition~\ref{d_co} emphasizes that coherence is really
a property of the column space, hence basis-independent.  In other
words, if $\hat{Q}=QV$, where $V$ is a real $n\times n$ orthogonal
matrix, then $\hat{Q}$ and $Q$ have the same coherence.

The range for coherence is $\tfrac{n}{m}\leq \mu\leq 1$.  If $Q$ is a
$m\times n$ submatrix of the $m\times m$ Hadamard matrix, then
$\mu=n/m$.  If a column of $Q$ is a canonical vector, then $\mu=1$.
Hence an orthonormal basis has high coherence if it shares a vector
with a canonical basis.

There are other
definitions of coherence that differ from the above by factors
depending on the matrix dimensions \cite[Definition 1]{Recht11},
\cite[Definition 1]{TalR10}. However, the notion of 
\textit{statistical coherence} in Bayesian analysis
\cite{Lind78} appears to be unrelated.

\section{Sampling Methods}\label{s_samp}
We present three different types of sampling methods: 
Sampling without replacement
(Section~\ref{s_without}),
sampling with replacement  (Section~\ref{s_with}),
and Bernoulli sampling (Section~\ref{s_bernoulli}).
We show that Bernoulli sampling 
can be viewed as a form of sampling
without replacement (Section~\ref{s_relate}).
The numerical experiments   illustrate
that there is little difference among the three methods
for small to moderate amounts of sampling (Section~\ref{s_sampcomp}). 
Hence we recommend sampling with replacement 
for Algorithm~\ref{alg_bpik} (Section~\ref{s_sampconc}).

The sampling matrices $S$ in all three methods are scaled so that $S^TS$
is an unbiased estimator of the identity matrix.

\subsection{Sampling without replacement}\label{s_without}
The obvious sampling strategy, in Algorithm~\ref{alg_without},
picks the requested number of
rows, so that the sampling matrix~$S$ is just a scaled submatrix of a 
permutation matrix.

Uniform sampling without replacement can be implemented via 
\textit{random permutations}\footnote{We thank an
anonymous reviewer for this advice.}.
A permutation $\pi_1,\ldots, \pi_m$ of the integers $1,\ldots, m$
is a \textit{random permutation}, if it is equally likely to be one of $m!$
possible permutations \cite[pages 41 and 48]{MitzUpf}.

\begin{algorithm}
\caption{Uniform sampling without replacement
\cite{GT11,GrN10}}\label{alg_without}
\begin{algorithmic}
\REQUIRE Integers $m\geq 1$ and $1\leq c\leq m$
\ENSURE $c\times m$ sampling matrix $S$ with $\E[S^TS]=I_m$
\STATE
\STATE Let $k_1,\ldots,k_m$ be a random permutation of $1,\ldots, m$
\smallskip

\STATE $S=\sqrt{\tfrac{m}{c}}\>
\begin{pmatrix}e_{k_1} & \ldots & e_{k_c}\end{pmatrix}^T$
\end{algorithmic}
\end{algorithm}

The following lemma presents the probability that sampling without replacement 
picks a particular row.

\begin{lemma}\label{l_without}
If  Algorithm~\ref{alg_without} samples $c$ out of $m$ indices, then the
probability that a particular index is picked equals $c/m$.
\end{lemma}

\begin{proof}
The probability that some index, say~$r$, is not sampled in
the first trial is $1-\tfrac{1}{m}=\tfrac{m-1}{m}$. 
Now there are only $m-1$ indices left.
So the probability that index~$r$ is not sampled in the second trial is
$1-\tfrac{1}{m-1}=\tfrac{m-2}{m-1}$.  Repeating this argument shows 
that with probability $\prod_{t=1}^c\frac{m-t}{m-t+1} = \frac{m-c}{m}$
index~$r$ is not sampled in $c$ trials.

The complementary event, the probability that index~$r$ is sampled,
equals $1-\tfrac{m-c}{m} = \tfrac{c}{m}$.
\end{proof}

\subsection{Sampling with replacement}\label{s_with}
This is the sampling strategy that appears to be analyzed in~\cite{AMTol10}. 
It samples exactly the requested number of rows, but with replacement,
which means a row may be sampled more than once.
Algorithm~\ref{alg_with}  is  the same
as the \textsl{EXACTLY(c)} algorithm \cite[Algorithm 3]{DMMS10}
with uniform probabilities,
which is also used in the \textsl{BasicMatrixMultiplication Algorithm}
\cite[Fig. 2]{DKM06}. 

\begin{algorithm}
\caption{Uniform sampling with replacement \cite{DKM06,DMMS10}}\label{alg_with}
\begin{algorithmic}
\REQUIRE Integers $m\geq 1$ and $1\leq c\leq m$
\ENSURE $c\times m$ sampling matrix $S$ with $\E[S^TS]=I_m$
\STATE
\FOR{$t = 1 : c$}
\STATE Sample $k_t$ from $\{1,\ldots,m\}$ with probability $1/m$,\\
\STATE independently and with replacement\\
\ENDFOR
\smallskip

\STATE $S=\sqrt{\tfrac{m}{c}}\>
\begin{pmatrix}e_{k_1} & \ldots & e_{k_c}\end{pmatrix}^T$
\end{algorithmic}
\end{algorithm}

Sampling with replacement (Algorithm~\ref{alg_with}) is often easier
to analyze and implement than sampling without replacement
(Algorithm~\ref{alg_without}), and it can also
be more robust to errors \cite[\S 1.2]{MitzUpf}.

\subsection{Bernoulli sampling}\label{s_bernoulli}
The sampling strategy in Algorithm~\ref{alg_bernoulli}
is implemented in \textsl{Blendenpik}  \cite[Algorithm 1]{AMTol10}.  
Following \cite[Section A]{GrN10}, we use the term
\textsl{Bernoulli sampling}, because the strategy
treats each row as an independent, identically 
distributed Bernoulli random variable. Each row is either sampled or not, 
with the same probability for each row.
Algorithm~\ref{alg_bernoulli} produces a 
$m\times m$ square matrix $S$ --
in contrast to Algorithms~\ref{alg_without} and~\ref{alg_with},
which produce $c\times m$ matrices.

\begin{algorithm}
\caption{Bernoulli sampling
\cite{AMTol10,GT11,GrN10}}\label{alg_bernoulli}
\begin{algorithmic}
\REQUIRE Integers  $m\geq 1$ and $1\leq c\leq m$ 
\ENSURE $m\times m$ sampling matrix $S$ with $\E[S^TS]=  I_m$
\STATE
\STATE $S=0_{m\times m}$   
\FOR{$t = 1 : m$}
\STATE \medskip
$S_{tt} = \sqrt{\tfrac{m}{c}}\>\begin{cases}
1 & \text{with probability $\tfrac{c}{m}$} \\
0 & \text{with probability  $1-\tfrac{c}{m}$} 
\end{cases}$
\ENDFOR
\end{algorithmic}
\end{algorithm}

The number of sampled rows, which is equal to the number of non-zero
diagonal elements in $S$, is not known a priori, but the expected
number of sampled rows is~$c$. The lemma below shows that the actual number of
rows picked by Bernoulli sampling is characterized by a binomial
distribution \cite[Section 2.2.2]{Ross}.

\begin{lemma}\label{l_bernoulli}
If Algorithm~\ref{alg_bernoulli} samples from $m$ indices with probability
$\gamma\equiv c/m$, then
the probability that it picks exactly $k$ indices equals
$\binom{m}{k}\>\gamma^{k}\>(1-\gamma)^{m-k}$.
\end{lemma}

\begin{proof}
Determining the diagonal elements of the $m\times m$ sampling matrix
$S$ in Algorithm~\ref{alg_bernoulli}
can be viewed as performing $m$ independent trials, where trial $t$ is
a success ($S_{tt}\neq 0$) with probability $\gamma$,
and a failure ($S_{tt}=0$) with probability $1-\gamma$.
The probability of $k$ successes is given by the binomial distribution 
$\binom{m}{k}\>\gamma^{k}\>(1-\gamma)^{m-k}$.
\end{proof}

\subsection{Relating Bernoulli sampling and sampling without 
replacement}\label{s_relate}
We show  that Bernoulli sampling (Algorithm~\ref{alg_bernoulli})
is the same as first determining the
number of samples with a binomial distribution
(motivated by Lemma~\ref{l_bernoulli}), and then sampling
without replacement (Algorithm~\ref{alg_without}). This is
described in Algorithm~\ref{alg_equiv} below.

\begin{algorithm}
\caption{Simulating Algorithm~\ref{alg_bernoulli} with 
Algorithm~\ref{alg_without}}\label{alg_equiv}
\begin{algorithmic}
\REQUIRE Integers  $m\geq 1$ and $1\leq c\leq m$ 
\ENSURE $\tilde{c}\times m$ sampling matrix $S$ with $\E[S^TS]=  I_m$\\
$\qquad\quad$ that ``behaves like'' a sampling matrix
generated by Algorithm~\ref{alg_bernoulli}
\STATE
\STATE $\gamma\equiv c/m$
\smallskip

\STATE Sample $\tilde{c}$ from $\{1,\ldots,m\}$ where
$\Prob[\tilde{c}=k] = \binom{m}{k}\>\gamma^k\>(1-\gamma)^{m-k}$
\smallskip

\STATE Use Algorithm~\ref{alg_without} to sample $\tilde{c}$ 
indices $k_1,\ldots,k_{\tilde{c}}$ uniformly and without 
replacement
\smallskip

\STATE $S=\sqrt{\tfrac{m}{\tilde{c}}}\>
\begin{pmatrix}e_{k_1} & \ldots & e_{k_{\tilde{c}}}\end{pmatrix}^T$
\end{algorithmic}
\end{algorithm}

Below we describe the sense in which Algorithm~\ref{alg_equiv}
``behaves like'' Bernoulli sampling in Algorithm~\ref{alg_bernoulli}.

\begin{lemma}
The probability that Algorithm~\ref{alg_equiv} picks a particular
index equals $\gamma=c/m$.
\end{lemma}

\begin{proof}
Motivated by Lemma~\ref{l_bernoulli},
the actual number of samples $k$ in Algorithm~\ref{alg_equiv} 
is given by a binomial distribution. Once a specific $k$ has emerged,
one applies Lemma~\ref{l_without} to conclude that the probability that
Algorithm~\ref{alg_without} picks some index $r$ is $k/m$.

Now the probability that Algorithm~\ref{alg_equiv} picks some index $r$
is obtained by conditioning \cite[Section 3.5]{Ross}
on the number of samples, $k$, and equals
\begin{eqnarray*}
\lefteqn{\sum_{k=0}^m{\Prob\left[\text{$k$ indices sampled}\right]\>
\Prob\left[\text{index $r$ sampled} {\large{|}}\> 
\text{$k$ indices sampled}\right]}}\hspace{15pt}\\
& = &\sum_{k=1}^m{\binom{m}{k}\>\gamma^k\>\left(1-\gamma\right)^{m-k}
\>\frac{k}{m}}\\
& = &\gamma\>
\sum_{k=0}^{m-1}{\binom{m-1}{k}\gamma^{k}\left(1-\gamma\right)^{m-1-k}}
=\gamma\left(\gamma + (1-\gamma)\right)^{m-1}=\gamma,
\end{eqnarray*}
where the first equality follows from the zero summand for $k=0$.
\end{proof}

Finally, we can conclude that sampling with Algorithm~\ref{alg_equiv} is the 
same as sampling with Algorithm~\ref{alg_bernoulli}.

\begin{theorem}
Both,  Algorithms \ref{alg_equiv} and~\ref{alg_bernoulli} 
pick a particular set of indices $i_1, \ldots, i_c$ 
with probability $\gamma^c(1-\gamma)^{m-c}$.
\end{theorem}

\begin{proof}
The probability that Algorithm~\ref{alg_bernoulli} samples 
indices
$i_1, \ldots, i_c$ is equal to  $\gamma^c(1-\gamma)^{m-c}$.  

We show that the same is true for Algorithm~\ref{alg_equiv}.  
The choice of the sampling distribution in 
Algorithm~\ref{alg_equiv} implies that it samples $\tilde{c}=c$
indices with
probability $\binom{m}{c}\>\gamma^c\>(1-\gamma)^{m-c}$.
Since there are $\binom{m}{c}$ ways to sample $c$ out of $m$ indices, 
the probbility that the particular index set $i_1, \ldots, i_c$ 
is picked, given that $c$
indices are being sampled is $1/\binom{m}{c}$.  Thus, the probability that
Algorithm~\ref{alg_equiv} picks indices $i_1, \ldots, i_c$ equals
$$\frac{1}{\binom{m}{c}}\>
\binom{m}{c}\>\gamma^c(1-\gamma)^{m-c} = \gamma^c(1-\gamma)^{m-c}.$$
\end{proof}

\subsection{Numerical experiments}\label{s_sampcomp}
We present two representative comparisons of the three sampling
strategies, with two plots for each strategy: The condition numbers of
full-rank sampled matrices $SQ$, and the failure percentage, that is
the percentage of sampled matrices $SQ$ that are numerically rank
deficient (as determined by the Matlab command \texttt{rank}).

The experiments are limited to very tall and skinny matrices
(with many more rows than columns, $m\gg n$), because
that's when the sampling strategies are most efficient. In particular,
since $c\geq n$ is required for $SQ$ to have full column rank,
sampling methods are inefficient when $n$ is not much smaller than
$m$, in which case a deterministic algorithm would be preferable.

\subsubsection*{Experimental setup}
The $m\times n$ matrices $Q$ with orthonormal columns have $m=10^4$ 
rows and $n=5$ columns. The condition numbers and failure percentages
are plotted against various sampling amounts $c$, with 30 runs for each $c$. 
For the failure percentages we display
only those sampling amounts $c$ that give rise to rank-deficient
matrices, in these particular 30 runs.
For Algorithm~\ref{alg_bernoulli} the horizontal
axis represents the numerator $c$ in the probability, that is, the expected 
number of sampled rows.
All three strategies sample from the same matrix.

We consider two different types of matrices:
Matrices with low coherence $\mu=1.5n/m$ in Figure~\ref{f_fig1}; and 
matrices with higher coherence $\mu=150n/m$ and many zero rows
in Figure~\ref{f_fig2}. Our numerical experiments indicate that these 
coherence values are representative, in the sense that different values 
of coherence would not produce any other interesting effects.

\begin{figure}  
\begin{center}
\resizebox{4.3in}{!}
{\includegraphics{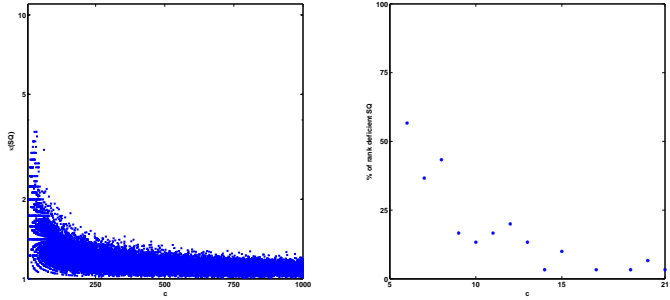}} \\
{\klein{(a)\ Algorithm~\ref{alg_without}: Sampling without replacement}}
\end{center}
\begin{center}
\resizebox{4.3in}{!}
{\includegraphics{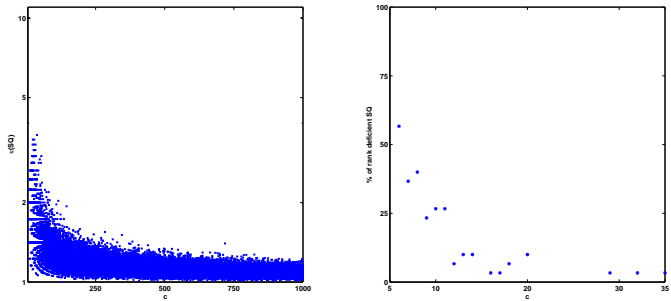}} \\
{\klein{(b)\ Algorithm~\ref{alg_with}: Sampling with replacement}}
\end{center}
\begin{center}
\resizebox{4.3in}{!}
{\includegraphics{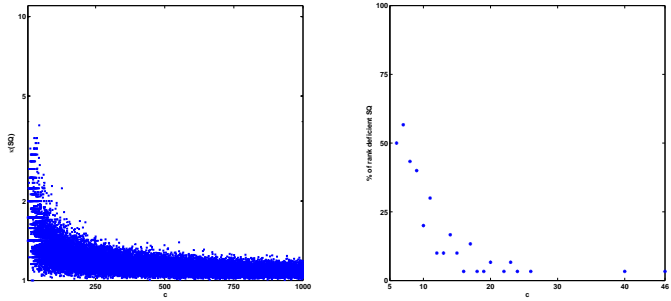}} \\
{\klein{(c)\ Algorithm~\ref{alg_bernoulli}: Bernoulli sampling}}
\end{center}
\caption{Condition numbers and percentage of rank-deficiency for matrices
with low coherence and small amounts of sampling. 
Here $Q$ is $m\times n$ with orthonormal 
columns, $m=10,000$, $n=5$, coherence $\mu=1.5n/m$, and generated with
Algorithm~\ref{alg_geno}.
Left panels: Horizontal coordinate axes represent amounts of sampling 
$n\leq c\leq 1,000$. Vertical coordinate axes represent
condition numbers $\kappa(SQ)$; the maximum is 10.
Right panels: Horizontal coordinate axes represent amounts of sampling that
give rise to numerically rank deficient matrices $SQ$.  
Vertical coordinate axes represent percentage of
numerically rank deficient matrices.
}\label{f_fig1}
\end{figure}

\subsubsection*{Figure~\ref{f_fig1}} Shown are condition numbers and 
percentage of rank deficient matrices for 
a matrix $Q$ with low coherence $\mu=1.5n/m$
generated by Algorithm~\ref{alg_geno}. At most 
10~percent of the rows are sampled.
The three strategies exhibit almost identical behavior:
The sampled matrices $SQ$ of full rank are very well
conditioned, with $\kappa(SQ)\leq 5$. Numerically 
rank-deficient matrices $SQ$ occur only for sampling amounts  $c\leq 47$.

\begin{figure}  
\begin{center}
\resizebox{4.3in}{!}
{\includegraphics{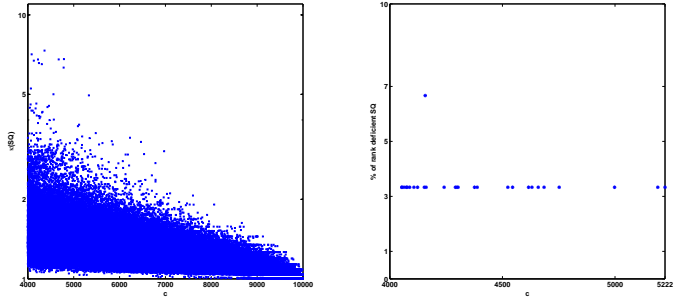}} \\
{\klein{(a)\ Algorithm~\ref{alg_without}: Sampling without replacement}}
\end{center}
\begin{center}
\resizebox{4.3in}{!}
{\includegraphics{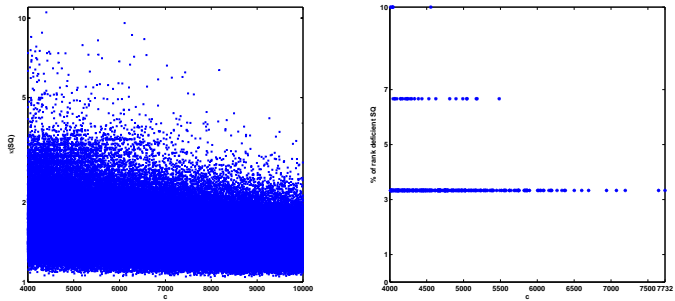}} \\
{\klein{(b)\ Algorithm~\ref{alg_with}: Sampling with replacement}}
\end{center}
\begin{center}
\resizebox{4.3in}{!}
{\includegraphics{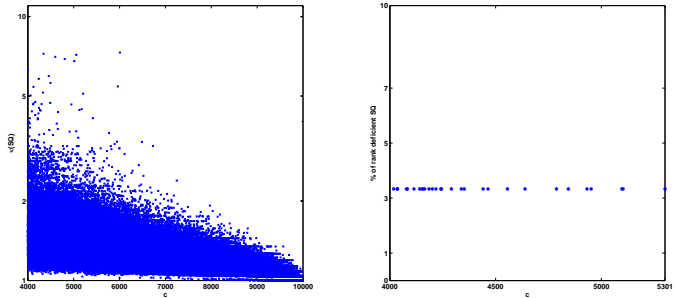}} \\
{\klein{(c)\ Algorithm~\ref{alg_bernoulli}: Bernoulli sampling}}
\end{center}
\caption{Condition numbers and percentage of rank-deficiency for matrices
with higher coherence and large amounts of sampling. 
Here $Q$ is $m\times n$ with orthonormal 
columns, $m=10,000$, $n=5$, coherence $\mu=150n/m$, and generated with
Algorithm~\ref{alg_zero}.
Left panels: Horizontal coordinate axes represent amounts of sampling 
$4,000\leq c\leq m$. Vertical coordinate axes represent
condition numbers $\kappa(SQ)$; the maximum is 10.
Right panels: Horizontal coordinate axes represent amounts of sampling that
give rise to numerically rank deficient matrices $SQ$.  
Vertical coordinate axes represent percentage of
numerically rank deficient matrices $SQ$; the maximum is 10 percent.
}\label{f_fig2}
\end{figure}

\subsubsection*{Figure~\ref{f_fig2}} 
Shown are condition numbers and percentage
of rank deficient matrices for 
a matrix $Q$, generated by Algorithm~\ref{alg_zero},
with coherence $150n/m$ and many zero rows.
The number of sampled rows ranges from $c=4000$ to $m$.
The sampled matrices $SQ$ of full rank are very well
conditioned, with $\kappa(SQ)\leq 10$.
Even for $c=4,000$, as many 10 percent of the sampled matrices can
still be rank-deficient.
All three algorithms have to sample
more than half of the rows of $Q$ in order to always produce 
matrices $SQ$ with full column rank. Specifically in these particular runs,
Algorithms \ref{alg_without} and~\ref{alg_bernoulli} need to sample
$c\geq 5,222$ and $c\geq 5,301$ rows, respectively, while
Algorithm~\ref{alg_without} needs $c\geq 7732$.

Note that the condition numbers of matrices from Algorithms
\ref{alg_without} and~\ref{alg_bernoulli} approach~1 as more and more
rows are sampled. This is because 
no row is sampled more than once; and for $c=m$ all rows are sampled.

Again, the three strategies exhibit almost identical behavior: The sampled
matrices $SQ$ of full rank are very well conditioned, with $\kappa(SQ)\leq 10$.
However, due to the higher coherence, numerically
rank-deficient matrices occur more frequently.

\subsection{Conclusions for Section~\ref{s_samp}}\label{s_sampconc}
The numerical experiments illustrate that the
three sampling strategies behave almost identically, in particular
for small to moderate
sampling amounts, and that sampled matrices of full rank tend to be
very well-conditioned\footnote{We have not been able to show rigorously
why the condition numbers tend to be less than~10.}. 
Furthermore, Section~\ref{s_relate} shows that Bernoulli
sampling can be viewed as a form of sampling without replacement,
and the numerical experiments confirm the similarity in behavior.

Among the three strategies, we recommend sampling with replacement
(Algorithm~\ref{s_with}) for small to moderate amounts of sampling
in Algorithm~\ref{alg_bpik}.
It is fast and easy to implement in both.

\section{Condition number bounds based on coherence}\label{s_cohbound}
We derive bounds for the condition numbers of matrices produced by the
sampling strategies in section~\ref{s_samp},
in terms of coherence. These bounds are based on a specific
concentration inequality and imply a, not necessarily tight, lower bound
for the number of sampled rows 
(Section~\ref{s_subcohbound}). Numerical experiments illustrate that
the bounds are informative (Section~\ref{s_cohex}). We end this 
section by summarizing the main features of the bounds 
(Section~\ref{s_cohconc}).

\subsection{Bounds}\label{s_subcohbound}
We show that the three sampling strategies in Section~\ref{s_samp}
all have the same condition number bound, in terms of coherence. 

Theorem~\ref{t_cohbound} below is based on a
matrix Chernoff concentration inequality (Section~\ref{s_chernoff}). 
We chose this particular inequality because extensive numerical experiments
with our Matlab toolbox \texttt{kappaSQ\_v3} \cite{impl}
suggest that it tends to produce the tightest bound.

\begin{theorem}\label{t_cohbound}
Let $Q$ be a real $m\times n$ matrix with $Q^TQ=I_n$ and coherence~$\mu$.
Let $S$ be a sampling matrix produced by 
Algorithms~\ref{alg_without}, \ref{alg_with}, or \ref{alg_bernoulli}
with $n\leq c\leq m$.
For $0<\epsilon<1$ and $f(x)\equiv e^x(1+x)^{-(1+x)}$ define
$$\delta\equiv n\left(f(-\epsilon)^{c/(m\mu)}+
f(\epsilon)^{c/(m\mu)}\right).$$
If $\delta<1$, then with probability at least $1-\delta$ we have
$\rank(SQ)=n$ and
$$\kappa(SQ)\leq \sqrt{\frac{1+\epsilon}{1-\epsilon}}.$$
\end{theorem}

\begin{proof} 
The proof is based on results from \cite{GT11,tropp11b,tropp11} and is 
relegated to Section~\ref{s_tcohproof}.
\end{proof}

Since $0<f(\pm\epsilon)<1$ for $0<\epsilon<1$,
Theorem~\ref{t_cohbound} implies that the sampling strategies
in Section~\ref{s_samp} are more likely to produce 
full-rank matrices as the number $c$ of sampled rows increases. Furthermore,
for a given total number of rows $m$, matrices $Q$ with fewer columns $n$
and lower coherence $\mu$ are more likely to give rise to 
sampled matrices $SQ$ that have full rank.

Theorem~\ref{t_cohbound} 
implies the following lower bound on the number of samples, but we make no
claims about the tightness of this bound.

\begin{corollary}\label{c_cohbound}
Under the assumptions of Theorem~\ref{c_cohbound},
$$c\geq 3m\mu\>\frac{\ln(2n/\delta)}{\epsilon^2}$$
samples are sufficient to achieve $\kappa(SQ)\leq \sqrt{\tfrac{1+\epsilon}{1-\epsilon}}$
with probability at least $1-\delta$.

\end{corollary}
\begin{proof} See Section~\ref{s_ccohproof}.
\end{proof}

Corollary~\ref{c_cohbound} implies that the sampling strategies
in Section~\ref{s_samp}
should sample at least $c=\Omega\left(m\mu\ln{n}\right)$ rows to 
produce a full rank, well-conditioned matrix.
In particular, if $Q$ has minimal coherence $\mu=n/m$, then 
Corollary~\ref{c_cohbound} implies that the number of sampled rows
should be at least
\begin{eqnarray}\label{e_nmcoh}
c\geq 3n \> \frac{\ln(2n/\delta)}{\epsilon^2},
\end{eqnarray}
that is $c=\Omega\left(n\ln{n}\right)$.

To achieve $\kappa(SQ)\leq 10$ with probability at least .99
requires that the number of sampled rows be at least
\begin{eqnarray}\label{c_lb}
c\geq 3.2\,m\mu\>\left(\ln(2n)+4.7\right).
\end{eqnarray}
Here we chose $\epsilon_0=99/101$, so that the condition number bound
equals $\sqrt{\frac{1+\epsilon_0}{1-\epsilon_0}}=10$.

\begin{remark}\label{r_lowcoh}
Theorem~\ref{t_cohbound} is informative only for sufficiently low
coherence values.

For instance, consider the higher coherence
matrices from Figure~\ref{f_fig2} 
in Section~\ref{s_sampcomp} with $m=10,000$, $n=5$ and coherence
$\mu =150n/m$. Choose $\epsilon=99/101$ 
so that $\kappa(SQ)\leq 10$, 
and a failure probability $\delta = .01$. Then
Corollary~\ref{c_cohbound} implies the lower bound
$c\geq 12,408$, which means that the number of sampled rows would
have to be larger than the total number of rows.
\end{remark}

\subsection{Numerical experiments}\label{s_cohex}
We compare the bound for the condition numbers
of the sampled matrices (Theorem~\ref{t_cohbound}) with the true 
condition numbers of matrices produced by sampling with replacement 
(Algorithm~\ref{alg_with}). 

There are several reasons why
it suffices to consider only a single sampling strategy:
The three sampling methods all have
the same bound (Theorem~\ref{t_cohbound}); Bernoulli sampling
is a form of sampling without replacement (Section~\ref{s_relate});
and all three sampling methods exhibit very similar behavior
for matrices of low coherence
(Sections \ref{s_sampcomp} and~\ref{s_sampconc}).
Furthermore, this allows a clean comparison with the bounds
in Section~\ref{s_leverage} which apply only to Algorithm~\ref{alg_with}.

\subsubsection*{Experimental setup}
The $m\times n$ matrices $Q$ with orthonormal columns have $m=10^4$
rows and $n=5$ columns. 
The left panels in Figure~\ref{f_fig3} show the
condition numbers of the full-rank sampled matrices $SQ$ produced by 
Algorithm~\ref{alg_with}
against different sampling amounts $c$, with 30 runs for each $c$.
The right panels in Figure~\ref{f_fig3} show the percentage of rank
deficient matrices $SQ$ against different sampling amounts~$c$.  We
display only those sampling amounts $c$ that give rise to
rank-deficient matrices, in these particular 30 runs.

The left panels in Figure~\ref{f_fig3} also show
the condition number bound 
$\kappa_{\epsilon}\equiv \sqrt{\tfrac{1+\epsilon}{1-\epsilon}}$
from Theorem~\ref{t_cohbound}. For each value of $c$, 
we obtain $\epsilon$ as
the solution of the nonlinear equation $F_c(x)^2=0$ associated with
Theorem~\ref{t_cohbound} and defined as 
$$F_c(x)\equiv \delta -
n\left(f(-x)^{c/(m\mu)}+ f(x)^{c/(m\mu)}\right).$$
We impose the stringent requirement
of $\delta = .01$, corresponding to a 99 percent success probability.
Since an explicit expression seems out of reach, we use unconstrained 
nonlinear optimization (a Nelder-Mead simplex direct search) to solve
$F_c(x)^2=0$. This is done in Matlab with a code equivalent to 
$$\epsilon=\left|\texttt{fminsearch}(F_c(x)^2,0, 10^{-30})\right|,$$
where \texttt{fminsearch} starts at the point~0, and  
terminates when $|F_c(\epsilon)|^2\leq 10^{-30}$.
If $0<\epsilon<1$  then 
$\kappa_{\epsilon}$ is plotted, otherwise nothing is plotted.

As explained in Remark~\ref{r_lowcoh}, 
Theorem~\ref{t_cohbound} is not informative for higher coherence 
values, so we consider matrices with the following properties:
Minimal coherence $\mu=n/m$ in
Figure~\ref{f_fig3}(a); low coherence $\mu=1.5n/m$ in
Figure~\ref{f_fig3}(b);  slightly higher coherence $\mu=15n/m$ 
with many zero rows in Figure~\ref{f_fig3}(c). 
The matrices for Figures~\ref{f_fig3}(a) 
and~\ref{f_fig3}(b) were generated with Algorithm~\ref{alg_geno}, 
while the matrix for Figure~\ref{f_fig3}(c) was generated 
with Algorithm~\ref{alg_zero}.

\subsubsection*{Figure~\ref{f_fig3}}
The left panels illustrate that Theorem~\ref{t_cohbound},
constrained to a 99 percent success probability, correctly
predicts the magnitude of the condition numbers, i.e.  $\kappa(SQ)\leq
10$.  Hence Theorem~\ref{t_cohbound} provides informative qualitative
bounds for matrices with very low coherence, as well as for matrices
with slightly higher coherence and many zero rows.

\begin{figure}  
\begin{center}
\resizebox{4in}{!}
{\includegraphics{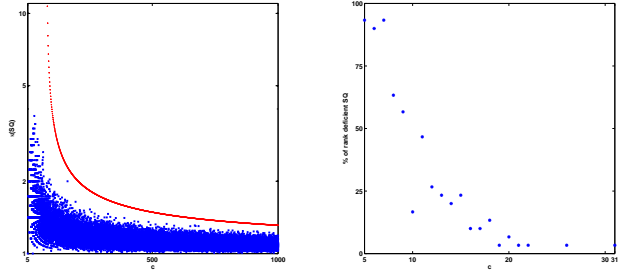}}\\
{\klein{(a)\ $Q$ has minimal coherence $\mu = n/m$. Sampling amounts
are $n\leq c\leq 1,000$.}}
\end{center}
\begin{center}
\resizebox{4in}{!}
{\includegraphics{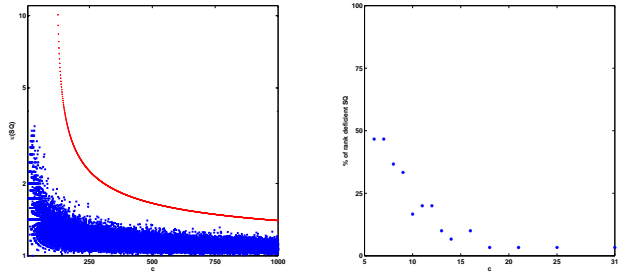}}\\
{\klein{(b)\ $Q$ has low coherence $\mu = 1.5 n/m$. Sampling amounts
are $n\leq c\leq 1,000$.}}
\end{center}
\begin{center}
\resizebox{4in}{!}
{\includegraphics{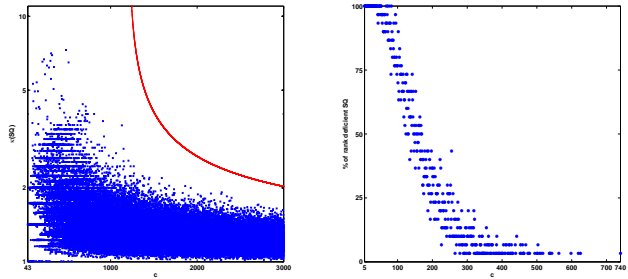}}\\
{\klein{(c)\ $Q$ has slightly higher coherence $\mu = 15n/m$ and many zero 
rows. Sampling sampling amounts are $n\leq c\leq 3,000$.}}
\end{center}
\caption{Condition numbers and bound from Theorem~\ref{t_cohbound},
and percentage of rank-deficiency.
Here $Q$ is $m\times n$ with orthonormal 
columns, $m=10,000$ and $n=5$.
Left panels: The horizontal coordinate axes represent amounts of sampling~$c$. 
The vertical coordinate axes represent
condition numbers $\kappa(SQ)$; the maximum is 10.
The dots at the bottom represent the
condition numbers of matrices sampled with Algorithm~\ref{alg_with}, 
while the upper line
represents the bound from Theorem~\ref{t_cohbound}.
Right panels: The horizontal coordinate axes represent amounts of sampling that
produce numerically rank deficient matrices $SQ$.  
The vertical coordinate axes represent the percentage of
numerically rank deficient matrices~$SQ$.
}\label{f_fig3}
\end{figure}

\subsubsection*{Table~\ref{t_tab1}}
This is a comparison of the numerical experiments in Figure~\ref{f_fig3}
with the bounds from Theorem~\ref{t_cohbound} and 
Corollary~\ref{c_cohbound}, both restricted to a 99 percent success
probability.

The third column depicts the highest values of $c$ for which
a rank-deficient matrix occurs, during these particular 30 runs. 
It should be kept in mind that these values are highly dependent on
the particular sampling runs. This column is to be compared 
to the fourth column which contains the lowest values of $c$ where
Theorem~\ref{t_cohbound} starts to apply.
Although there is a gap between the occurrence of the last rank deficiency
and the onset of Theorem~\ref{t_cohbound}, the values have qualitatively
the same order of magnitude.

The rightmost column in Table~\ref{t_tab1} contains the values of the
lower bound (\ref{c_lb}), and is to be compared to the column 
with the starting values for Theorem~\ref{t_cohbound}.  Although
(\ref{c_lb}) is weaker than Theorem~\ref{t_cohbound}, its values are
close to the starting values of Theorem~\ref{t_cohbound},
especially for lower coherence.
Hence, the lower bound (\ref{c_lb}) captures
the correct magnitude of the sampling amounts where Theorem starts to
become informative.

Table~\ref{t_tab1} illustrates that, although Theorem~\ref{t_cohbound}
and Corollary~\ref{c_cohbound} tend to become more pessimistic with
increasing coherence, they still provide qualitative information for
matrices with low coherence -- even when restricted to a 99 percent
success probability.

\begin{table}
\begin{center}
\begin{tabular}{|l||c|c|c|c|}
\hline
Figure& coherence $\mu$ & last rank deficiency & Theorem~\ref{t_cohbound} 
& (\ref{c_lb})\\
& & occurs at $c=$ & starts at $c=$ & \\
\hline\hline
\ref{f_fig3}(a) & $n/m$ & 31 & 81 & 83 \\
\ref{f_fig3}(b) & $1.5\>n/m$ & 31 & 121 & 125\\
\ref{f_fig3}(c) & $15\>n/m$ & 740 & 1207 & 1241\\
\hline
\end{tabular}
\end{center}
\bigskip
\caption{Comparison of information from Figure~\ref{f_fig3},
with Theorem~\ref{t_cohbound} and 
Corollary~\ref{c_cohbound}.}\label{t_tab1}
\end{table}

\subsection{Conclusions for Section~\ref{s_cohbound}}\label{s_cohconc}
The bounds in Theorem~\ref{t_cohbound} and Corollary~\ref{c_cohbound}
have the following advantages:

\begin{enumerate}
\item They are non-asymptotic bounds, where all constants have explicit
numerical values, hence they are tighter than 
the bounds in \cite[Theorem 3.2]{AMTol10}.

\item They apply to three different sampling methods.
\item They imply a lower bound, of $\Omega\left(m\mu\ln{n}\right)$,
on the required number of sampled rows. Although we did not give a formal
proof of tightness, 
numerical experiments illustrate that sampling only the required number
of rows implied by the bound is realistic.
numerical experiments illustrate that the bound is realistic.

\item Even under the stringent requirement of a 99 percent success 
probability, 
they are informative for matrices of small dimension because they
correctly predict the magnitude of the condition
numbers for the sampled matrices.
\end{enumerate}

Note that the bounds in Theorem~\ref{t_cohbound} and
Corollary~\ref{c_cohbound} are informative only for matrices that are
tall and skinny ($m\gg n$) and have low coherence. 
The restriction to tall and skinny matrices is not an imposition,
because it is required for 
the effectiveness of the sampling strategies, see Section~\ref{s_sampcomp}.

In the next section we try to relax the restriction to low coherence
matrices, by more thoroughly exploiting the information available from
the row norms of $Q$.

\section{Condition number bounds based on leverage scores, for uniform sampling 
with replacement}\label{s_leverage} 

The goal is to tighten Theorem~\ref{t_cohbound} by making use of all
the row norms of $Q$, instead of just the largest one.
To this end we introduce leverage scores
(Section~\ref{s_levscore}), which are the squared row norms of $Q$.
We use them to derive a bound for uniform sampling with
replacement (Section~\ref{s_levbound}), and for
more easily computable versions of the bound (Section~\ref{s_clevbound}).
Analytical
(Section~\ref{s_ancomp}) and experimental (Section~\ref{s_numcomp})
comparisons demonstrate that the implied lower bound on the number of
sampled rows is better than the coherence-based bounds in
Section~\ref{s_cohbound}. A review with some reflection ends this
section (Section~\ref{s_lqconc}).

\subsection{Leverage scores}\label{s_levscore}
So-called \textit{statistical leverage scores} were first introduced
in 1978 by Hoaglin and Welsch \cite{HoagW78} to detect outliers when
computing regression diagnostics, see also
\cite{ChatterH86,VelleW81}. Mahoney and Drineas pioneered the use
of leverage scores for importance sampling strategies in randomized
matrix computations \cite{Mah11}.

Specifically, 
if $M$ is a real $m\times n$ matrix with $\rank(M)=n$, then the 
$m\times m$  \textit{hat matrix} 
$$H\equiv M(M^TM)^{-1}M^T$$ is the orthogonal
projector onto the column space of $M$, and its diagonal elements 
are called \textit{leverage scores} \cite[Section~2]{HoagW78}.
Hence, leverage scores are basis-independent. For our
purposes, though, it suffices to define them in terms of a thin QR
decomposition $M=QR$, so that the hat matrix can be expressed
as $H=QQ^T$. 

\begin{definition}
If $Q$ is a $m\times n$ matrix with $Q^TQ=I_n$, then its
{\rm leverage scores} are
$$\ell_j\equiv \|e_j^TQ\|_2^2, \qquad 1\leq j\leq m.$$
The $m\times m$ diagonal matrix of leverage scores is
$$L\equiv\diag(\ell_1,\ldots,\ell_m).$$
\end{definition}

Note that the coherence is the largest leverage score, 
$$\mu=\max_{1\leq j\leq m}{\ell_j}=\|L\|_2.$$ 

\subsection{Bounds}\label{s_levbound}
The bound in Theorem~\ref{t_levbound} below involves leverage scores and
is based on a matrix Bernstein 
concentration inequality (Section~\ref{s_bernstein}), rather than on 
the matrix Chernoff concentration inequality (Section~\ref{s_chernoff})
for Theorem~\ref{t_cohbound}. Although the Bernstein inequality may not always
be as tight, we did not see how to insert leverage scores into the 
Chernoff inequality.

\begin{theorem}\label{t_levbound}
Let $Q$ be a $m\times n$ real matrix with $Q^TQ=I_n$, leverage scores
$\ell_j$, $1\leq j\leq m$, and coherence $\mu$. Let $S$ be a sampling
matrix produced by Algorithm~\ref{alg_with} with
$n\leq c\leq m$. For  $0<\epsilon<1$ set
$$\delta \equiv 2n\exp\left(-\tfrac{3}{2}\>
\frac{c \epsilon^2}{m\>(3\|Q^TLQ\|_2+\epsilon\mu)}\right).$$
If $\delta<1$, then with probability at least $1-\delta$ we have
$\rank(SQ)=n$ and 
$$\kappa(SQ) \leq \sqrt{\frac{1+\epsilon}{1-\epsilon}}.$$
\end{theorem}

\begin{proof} The proof uses results from  \cite{BalzRN11,Recht11}
and is relegated to Section~\ref{s_tlevproof}.
\end{proof}

Like Theorem~\ref{t_cohbound}, Theorem~\ref{t_levbound} implies that
sampling with replacement is more likely to produce full-rank matrices
as the number $c$ of sampled rows increases. Furthermore, for a given
total number of rows $m$, matrices $Q$ with fewer columns $n$ and
lower coherence $\mu$ are more likely to yield sampled matrices
$SQ$ that have full rank. The 
dependence of $\|Q^TLQ\|_2$ on $\mu$ is discussed below.

\begin{remark} \label{r_lq}
The norm $\|Q^TLQ\|_2$ has simple and tight bounds in terms of the coherence,
\begin{eqnarray}\label{e_mu}
\mu^2\leq \|Q^TLQ\|_2\leq  \mu.
\end{eqnarray}

The lower bound follows from $\|Q^TLQ\|_2=\|L^{1/2}Q\|_2^2$ and 
$$\|L^{1/2}Q\|_2\geq \|e_j^TL^{1/2}Q\|_2=\ell_j^{1/2}\>\|e_j^TQ\|_2=\ell_j,
\qquad 1\leq j\leq m,$$
which implies  $\|L^{1/2}Q\|_2\geq\mu$.
\smallskip

The bounds (\ref{e_mu}) are attained for extreme values of the coherence:
\begin{itemize}
\item In case of minimal coherence $\mu=\ell_j$ for all $1\leq j\leq m$, 
we have
$L=\mu I_m$. Thus $\|Q^TLQ\|_2=\mu\|Q^TQ\|_2=\mu$, and the upper bound is 
attained.
\item In case of maximal coherence $\mu=1$, we have $\mu^2=\mu$.
Thus $\|Q^TLQ\|_2=\mu^2=\mu$, and both, lower and upper bounds
are attained.
\end{itemize}
\end{remark}

\subsection{Computable bounds}\label{s_clevbound}
We present easily computable bounds for $\|Q^TLQ\|_2$, based 
on coherence and several of the largest leverage scores.

To this end, we use a labeling of the leverage scores in non-increasing order,
$$\mu=\ell_{[1]}\geq \cdots\geq \ell_{[m]}.$$ 

\begin{corollary}\label{c_1}
Under the assumptions of Theorem~\ref{t_levbound}, if
$t\equiv\left\lfloor 1/\mu\right\rfloor$, then
$$\|Q^TLQ\|_2\leq \mu\>\sum_{j=1}^t{\ell_{[j]}}+(1-t\,\mu)\,\ell_{[t+1]}
\leq \mu.$$
If, in addition, $t$ is an integer, then
$\|Q^TLQ\|_2\leq \mu\>\sum_{j=1}^t{\ell_{[j]}}$.
\end{corollary}

\begin{proof}
See Section~\ref{s_lqproof}.
\end{proof}

The number of large leverage scores appearing in Corollary~\ref{c_1} depends on 
the coherence: Few leverage scores for high coherence, but more
for low coherence. Henceforth we will use the approximation
from Corollary~\ref{c_1} instead of the true value $\|Q^TLQ\|_2$, for two
reasons: First, numerical experiments show that the 
approximation tends to be very accurate. Second, the
approximation is convenient, because it requires 
only a leverage score distribution rather than a full-fledged matrix~$Q$.

\begin{remark}
Corollary~\ref{c_1} is tight for the extreme cases of minimal and maximal
coherence.

\begin{itemize}
\item In case of minimal coherence $\mu=\ell_j$ for all
$1\leq j\leq m$,  Remark~\ref{r_lq} implies $\|Q^TLQ\|_2=\mu$.
The bound in Corollary~\ref{c_1} is $\|Q^TLQ\|_2\leq \mu$, thus tight.

\item In case of maximal coherence $\mu=1$, Remark~\ref{r_lq} implies
$\|Q^TLQ\|_2= \mu^2=\mu$. 
Corollary~\ref{c_1} holds with $t=1$ and gives the
bound $\|Q^TLQ\|_2\leq \mu$, which is tight as well.
\end{itemize}
\end{remark}

Inserting this approximation for $\|Q^TLQ\|_2$ into the expression
for $\delta$ in Theorem~\ref{t_levbound} yields a, not necessarily tight,
lower bound on the number of samples. 

\begin{corollary}\label{c_levbound}
Under the assumptions of Theorem~\ref{t_levbound},
$$c\geq \tfrac{2}{3}m\>(3\tau+\epsilon\mu)\>
\frac{\ln(2n/\delta)}{\epsilon^2},$$
where $\tau\equiv \mu\>\sum_{j=1}^t{\ell_{[j]}}+(1-t\,\mu)\,\ell_{[t+1]}$, samples are sufficient to achieve $\kappa(SQ)\leq \sqrt{\tfrac{1+\epsilon}{1-\epsilon}}$
with probability at least $1-\delta$.
\end{corollary}

In particular, if $Q$ has minimal coherence $\mu=n/m$, then 
Corollary~\ref{c_levbound} implies that the number of sampled rows
should be at least
$$c\geq 3n \> \frac{\ln(2n/\delta)}{\epsilon^2}.$$
This is the same as the coherence-based lower bound (\ref{e_nmcoh}).

To achieve $\kappa(SQ)\leq 10$ with probability at least .99
requires that the number of sampled rows be at least
\begin{eqnarray}\label{c_lbl}
c\geq m\>(2.1\tau+.7\mu)\>\left(\ln(2n)+4.7\right).
\end{eqnarray}

\subsection{Analytical comparison of the bounds in Sections \ref{s_subcohbound}
and~\ref{s_levbound}}\label{s_ancomp}
An analytical comparison between Theorems \ref{t_cohbound} 
and~\ref{t_levbound} is not obvious, because they are based on different
concentration inequalities. Instead we compare the implied lower 
bounds for the number of sampled rows, and show that the
leverage-score based bound in Corollary~\ref{c_levbound}
is at least as tight as the coherence-based bound in Corollary~\ref{c_cohbound}.

\begin{corollary}\label{c_comp}
Under the assumptions of Theorem~\ref{t_levbound} and Corollary~\ref{c_1},
$$\tfrac{2}{3}m\>(3\tau+\epsilon\mu)\>
\frac{\ln(2n/\delta)}{\epsilon^2}\leq 
3m\mu\>\frac{\ln(2n/\delta)}{\epsilon^2}.$$
Hence Corollary~\ref{c_levbound} is at least as tight as
Corollary~\ref{c_cohbound}.
\end{corollary}

\begin{proof} See Section~\ref{s_compproof}.
\end{proof}

\subsection{Experimental comparison of the bounds in Sections 
\ref{s_subcohbound} and~\ref{s_levbound}}\label{s_numcomp}
We present numerical experiments to compare the lower bounds for the number 
of sampled rows in 
Corollaries~\ref{c_cohbound} and~\ref{c_levbound}, 
for different values of coherence. This gives quantitative insight
into the comparison in Corollary~\ref{c_comp}, and illustrates the 
reduction in the number of sampled rows from
Corollary~\ref{c_levbound}, as compared to Corollary~\ref{c_cohbound}.

\subsubsection*{Experimental setup}
As in previous sections, we use $m\times n$ matrices with
$m=10^4$ rows and $n=5$ columns. The success
probability is .99; and $\epsilon=99/101$, so that the bound for
$\kappa(SQ)$ is equal to 10. Hence the bounds in Corollaries
\ref{c_cohbound} and~\ref{c_levbound}  amount to (\ref{c_lb}) and
(\ref{c_lbl}), respectively.

We consider two different leverage scores distributions: A distribution 
generated by Algorithm~\ref{alg_geno} 
with one large leverage score in Table~\ref{t_tab2};
and a distribution generated by Algorithm~\ref{alg_zero} with as many zeros 
as possible in Table~\ref{t_tab3}.

\subsubsection*{Table~\ref{t_tab2}}
This table shows the lower bounds on the number of sampled rows, 
for a leverage score distribution 
generated with Algorithm~\ref{alg_geno} that consists of
one large leverage score, equal to the coherence,
and all remaining leverage scores being non-zero and identical.
The bounds, as well as the approximation $\tau$ to $\|Q^TLQ\|_2$,
are displayed for eight different values of coherence, ranging 
from minimal coherence $\mu=n/m$ to $\mu =100n/m$.

Table~\ref{t_tab2} illustrates that with increasing coherence, the
number of sampled rows implied by Corollary~\ref{c_levbound} is only
about 20 percent of that from Corollary~\ref{c_cohbound}. This is because 
$\tau$ increases much more slowly than $\mu$. For instance, 
 $\tau\approx \mu/10$ when $\mu=100n/m$.

\begin{table}
\begin{center}
\begin{tabular}{|c||c|c|c|c|c|c|c|c|}
\hline
$\mu/(n/m)$ & 1 & 5 & 10 & 15 & 20 &25 &50& 100\\
\hline\hline 
Cor. \ref{c_cohbound} & 108 & 540 & 1,079 & 1,618 & 2,157 &2,697 &5,393 &10,786\\
Cor. \ref{c_levbound} & 96  & 191 & 310   & 432   & 556   &682   &1,3343 & 2,777\\
\hline
$\tau/(n/m)$ & 1.00 & 1.01 & 1.04   & 1.10 &1.19 & 1.30 & 2.22 &9.95\\
\hline
\end{tabular}
\end{center}
\bigskip
\caption{Lower bounds for number of sampled rows in Corollaries
\ref{c_cohbound} and~\ref{c_levbound}
and approximation $\tau$, for different values of coherence~$\mu$.  
The first value represents minimal coherence $\mu=n/m$.
Here $m=10,000$, $n=5$, $\delta=.01$, $\epsilon=99/101$, with
leverage scores generated by Algorithm~\ref{alg_geno}.}\label{t_tab2}
\end{table}

\subsubsection*{Table~\ref{t_tab3}}
This table shows the lower bounds on the number of sampled rows. The corresponding
leverage score distribution is
generated with Algorithm~\ref{alg_zero} and consists of
as many zeros as possible. All non-zero leverage
scores, expect possibly one, are equal to the coherence $\mu$, so that 
$\tau\approx\mu$.
The bounds are displayed for eight different values of coherence, ranging 
from minimal coherence $\mu=n/m$ to $\mu =100n/m$.

The bounds for Corollary~\ref{c_cohbound} are the same as in 
Table~\ref{t_tab2}, because the coherence values are the same.
Since $\tau=\mu$, the difference between Corollaries \ref{c_cohbound}
and~\ref{c_levbound} is not as drastic as in Table~\ref{t_tab2}, yet
it increases with increasing coherence. For $\mu=100n/m$,
Corollary~\ref{c_levbound} remain informative, while
Corollary~\ref{c_cohbound} does not.

\begin{table}
\begin{center}
\begin{tabular}{|c||c|c|c|c|c|c|c|c|}
\hline
$\mu/(n/m)$ & 1 & 5 & 10 & 15 & 20 &25 &50& 100\\
\hline\hline 
Cor. \ref{c_cohbound} & 108 & 540 & 1,079 & 1,618 & 2,157&2,697 &5,393&10,787\\
Cor. \ref{c_levbound}& 96 & 477 & 954 & 1,431 & 1,908 &2,385 &4,770 & 9,539\\
\hline
\end{tabular}
\end{center}
\bigskip
\caption{Lower bounds for number of sampled rows
in Corollaries \ref{c_cohbound} and~\ref{c_levbound},
for different values of coherence~$\mu$.  
The first value represents minimal coherence $\mu=n/m$.
Here $m=10,000$, $n=5$, $\delta=.01$, $\epsilon=99/101$, with
leverage scores generated by Algorithm~\ref{alg_zero}.}\label{t_tab3}
\end{table}

\subsection{Conclusions for Section~\ref{s_leverage}}\label{s_lqconc}
The goal of this section was to derive condition number
bounds that are based on leverage scores
rather than just coherence, when rows are sampled uniformly with replacement 
(Algorithm~\ref{alg_with}).
Corollary~\ref{c_comp} and the numerical experiments illustrate that 
the lower bound on the number of sampled rows
implied by Corollary~\ref{c_levbound} is smaller than that from
Corollary~\ref{c_cohbound}. 

Although the coherence based bound in Theorem~\ref{t_cohbound}
is derived from a stronger concentration inequality than 
the one for Theorem~\ref{t_levbound},
this difference disappears in the weakening necessary to obtain 
lower bounds for the amount of sampling.  Even in cases when 
the leverage score measure $\tau$ is the same as the coherence,
Corollary~\ref{c_levbound} still retains a small advantage, which 
can increase with
increasing coherence. Hence Corollary~\ref{c_levbound} tends to remain 
informative
for larger values of coherence, even when Corollary~\ref{c_cohbound} fails.

The difference in implied sampling amounts becomes more drastic in the 
presence widely varying non-zero leverage scores,
and can be as high as ten percent. This is because the coherence-based 
bound in Corollary~\ref{c_cohbound} cannot take advantage of the distribution
of the leverage scores.

Hence, when it comes to lower bounds for the number of rows sampled
uniformly with replacement, we recommend Corollary~\ref{c_levbound}. 

We have yet to derive leverage score based bounds for the
other two sampling strategies, uniform sampling without replacement 
(Algorithm~\ref{alg_without}) and Bernoulli sampling 
(Algorithm~\ref{alg_bernoulli}).

\section{Algorithms for generating matrices with prescribed
coherence and leverage scores}\label{s_genalg}
In order to investigate the efficiency of the sampling methods in 
Section~\ref{s_samp}, and test 
the tightness of the bounds in Sections \ref{s_cohbound} and~\ref{s_leverage}, 
we need to generate matrices with 
orthonormal columns that have prescribed leverage scores and coherence. 
The algorithms are implemented in the Matlab 
package \textsl{kappa\_SQ\_v3} \cite{impl}.

We present algorithms for generating matrices with prescribed leverage scores
and coherence (Section~\ref{s_alg}), and for generating
particular leverage score distributions
with prescribed coherence (Section~\ref{s_precoh}).
Such distributions can then, in turn, serve as inputs for the algorithm 
in Section~\ref{s_alg}. Furthermore we present two classes 
of structured matrices with prescribed coherence that are easy and fast 
to generate (Section~\ref{s_mat}).

\subsection{Matrices with prescribed leverage scores}\label{s_alg}
We present an algorithm that generates matrices
with orthonormal columns that have prescribed leverage scores.
In Section~\ref{s_pre} we prove an existence 
result to show that this is always possible.

Algorithm~\ref{alg_gen} is a transposed version of \cite[Algorithm 3]{DHST05}.
It repeatedly applies $m\times m$ Givens rotations $G_{ij}$
that rotate two rows $i$ and $j$, and are computed from
numerically stable expressions \cite[section 3.1]{DHST05}.
At most $m-1$ such rotations are necessary.
Since each rotation affects only two rows, Algorithm~\ref{alg_gen}
requires $\mathcal{O}(mn)$ arithmetic operations.

\begin{algorithm}
\caption{Generating a matrix with prescribed leverage scores
\cite{DHST05}}\label{alg_gen}
\begin{algorithmic}
\REQUIRE Integers $m$ and $n$ with $m\geq n\geq 1$\\
$\qquad\ $ Vector $\ell$ with elements $0\leq \ell_1\leq\cdots
\leq \ell_m\leq 1$ 
and $\sum_{j=1}^m{\ell_j}=n$ 
\ENSURE $m\times n$ matrix $Q$ with $Q^TQ=I_n$ and 
leverage scores $\|e_j^TQ\|_2^2 =\ell_j$, $1\leq j\leq m$\\
$\qquad$\\
\STATE $Q=\begin{pmatrix}I_n & 0_{n\times (m-n)}\end{pmatrix}^T$
$\qquad$ \COMMENT{Initialization}
\REPEAT
\STATE Determine indices $i<k<j$ with\\
$\|e_i^TQ\|_2^2<\ell_i$, $\|e_k^TQ\|_2^2=\ell_k$, $\|e_j^TQ\|_2^2>\ell_j$
\IF{$\ell_i-\|e_i^TQ\|_2^2\leq \|e_j^TQ\|_2^2-\ell_j$}
\STATE Apply rotation $G_{ij}$ to rows $i$ and $j$ so that 
$\|e_i^TG_{ij}\,Q\|_2^2=\ell_i$
\ELSE 
\STATE Apply rotation $G_{ij}$ to rows $i$ and $j$ so that 
$\|e_j^TG_{ij}\,Q\|_2^2=\ell_j$
\ENDIF
\STATE $Q=G_{ij}\,Q$ $\qquad$ \COMMENT{Update}
\UNTIL{no more such indices exist}
\end{algorithmic}
\end{algorithm}

\subsection{Leverage score distributions with prescribed 
coherence}\label{s_precoh}
We present algorithms that generate leverage score distributions for 
prescribed coherence. The resulting distributions then serve 
as inputs for Algorithm~\ref{alg_gen}.
These particular leverage score distribution help to distinguish the effect 
of coherence, which is the largest leverage score, from that of the
remaining leverage scores.

\subsubsection*{One large leverage score}
Given a prescribed coherence $\mu$,
Algorithm~\ref{alg_geno} generates a distribution consisting of 
one large leverage score equal to $\mu$ and the remaining leverage scores
being identical and non-zero. 

\begin{algorithm}
\caption{Generating a leverage score distribution with prescribed coherence:
One large leverage score}\label{alg_geno}
\begin{algorithmic}
\REQUIRE Integers $m$ and $n$ with $m\geq n\geq 1$\\
\STATE $\qquad\ $ Real number $\mu$ with $n/m\leq \mu\leq 1$\\
\ENSURE Vector $\ell$ with elements $\ell_1=\mu$, $0< \ell_j \leq 1$ 
and $\sum_{j=1}^m{\ell_j}=n$\\
$\qquad$\\
\STATE $\ell_1=\mu$ 
\FOR{$j=2:m$}
\STATE $\ell_j=\tfrac{n-\mu}{m-1}$ 
\ENDFOR
\end{algorithmic}
\end{algorithm}

In the special case of minimal coherence $\mu=n/m$,
Algorithm~\ref{alg_geno} generates $m$ identical leverages equal to
$\mu$, which is the only possible leverage score distribution in this case.

\subsubsection*{Many zero leverage scores}
Given a prescribed coherence, Algorithm~\ref{alg_zero}
generates a distribution with as many zero leverage scores as possible.
This serves as an ``adversarial'' distribution for the sampling algorithms
in Section~\ref{s_samp}. 

Given a prescribed coherence $\mu$, Algorithm~\ref{alg_zero} first determines 
the smallest number of rows $m_s$ that can realize this coherence, 
sets $m_s-1$ leverage scores equal to $\mu$, assigns another leverage score to
to take up the possibly non-zero slack, and sets the remaining 
leverage scores to zero.

\begin{algorithm}
\caption{Generating a leverage score distribution with prescribed coherence:
Many zero leverage scores}\label{alg_zero}
\begin{algorithmic}
\REQUIRE Integers $m$ and $n$ with $m\geq n\geq 1$\\
$\qquad\ $ Real number $\mu$ with $n/m\leq \mu\leq 1$\\
\ENSURE  Vector $\ell$ with elements $\ell_1=\mu$, $0\leq \ell_j \leq 1$ 
and $\sum_{j=1}^m{\ell_j}=n$\\
$\qquad$\\
\STATE $m_s=\lceil n/\mu\rceil$ $\qquad$ \COMMENT{Number of nonzero rows}
\FOR{$j=1:m_s-1$}
\STATE $\ell_j=\mu$ 
\ENDFOR
\STATE $\ell_{m_s}= n-(m_s-1)\,\mu$
\FOR{$j=m_s+1:m$}
\STATE $\ell_j=0$ 
\ENDFOR
\end{algorithmic}
\end{algorithm}

\subsection{Structured matrices with prescribed coherence}\label{s_mat}
We present two clas\-ses of structured matrices with orthonormal columns
that have prescribed coherence.  Although the structure puts constraints on the
matrix dimensions, the generation of these matrices is faster than
running Algorithm~\ref{alg_gen}. Note that the matrices produced by
Algorithm~\ref{alg_gen} also have structure, but it is not easily
characterized.

\paragraph{Stacks of diagonal matrices}
Given matrix dimensions $m$ and $n$, where $s=m/n$ is an integer,
and prescribed coherence $\mu$. The $m\times n$ matrix $Q$ below has
orthonormal columns and coherence $\mu$, and consists of
$s$ stacks of $n\times n$ diagonal matrices,
$$Q=
\begin{pmatrix}\sqrt{\mu}\,I_n\\ \phi\,  I_n\\ \vdots \\ \phi\, I_n\end{pmatrix}
\qquad where \qquad
\phi\equiv \sqrt{\frac{1-\mu}{\frac{m}{n}-1}}.$$

\paragraph{Matrices with Hadamard structure}
Given matrix dimensions $m$ and $n$, where $m=2^k$ and $n<m$ is also a power
of two, and prescribed coherence $\mu$. The $m \times n$ matrix
$$Q= D_k\begin{pmatrix}I_n\\0\end{pmatrix}$$
has orthonormal columns and coherence $\mu$, and is defined recursively
as follows.
For
$$\alpha\equiv\sqrt{\frac{\mu-\frac{n-1}{m-1}}{1-\frac{n-1}{m-1}}},\qquad
\beta\equiv \sqrt{\frac{1-\alpha^2}{m-1}}$$
define square matrices $B_j$ of dimension $2^j$ and square matrices $D_j$
of dimension $2^{j+1}$ as follows,
\begin{eqnarray*}
&B_0=\beta, \qquad 
&B_{j+1}=\begin{pmatrix} -B_j & B_j \\ B_j & B_j\end{pmatrix} \quad\qquad
0\leq j\leq k-1\\
&D_1=\begin{pmatrix}\alpha & -\beta \\ \beta & \alpha\end{pmatrix},\qquad
&D_{j+1}=\begin{pmatrix}D_j & -B_j \\ B_j & D_j\end{pmatrix}.
\end{eqnarray*}
Note that only the final matrix $Q$ has orthonormal columns and coherence
$\mu$ while, in general, the intermediate matrices $B_j$ and $D_j$ do not.
We omit the messy induction proof, because it does not provide much
insight.

\section{Future work}\label{s_future}
We have investigated three strategies for uniform sampling of rows from matrices 
with orthonormal columns: Without replacement, with replacement, and Bernoulli
sampling. We derived bounds on the condition numbers of the sampled
matrices, in terms of coherence and leverage scores. Numerical
experiments confirm that the bounds are realistic, even for 
high success probabilities and
matrices with small dimensions.

The following work still needs to be done.

\begin{itemize}
\item Conversion of the \textsl{kappa\_SQ\_v3} MATLAB toolbox
from a research code to a 
robust, flexible, and user-friendly GUI that facilitates 
reproducible research in the randomized algorithms community.

\item Tightening of Corollary~\ref{c_cohbound} so that it retains the 
strength of the Chernoff concentration inequality 
inherent in Theorem~\ref{t_levbound}.

\item Extension of the condition number bounds in Section~\ref{s_leverage} to 
 uniform sampling without replacement 
(Algorithm~\ref{alg_without}) and Bernoulli sampling 
(Algorithm~\ref{alg_bernoulli}). 

\item  Determination of a statistically significant number of runs
for each sampling amount $c$, for two purposes:
\begin{enumerate}
\item To assert, within a specific confidence interval, bounds on the 
condition numbers of the \textit{actually sampled} matrices.
\item To assert with a specific confidence that the probabilistic
expressions in Sections \ref{s_cohbound} and~\ref{s_leverage} 
do indeed represent bounds.
\end{enumerate}
\end{itemize}

\subsection*{Acknowledgements}
We are very grateful to John Holodnak, Petros Drineas and two anonymous
reviewers for reading our paper so 
carefully and for providing many helpful suggestions.

\appendix  
\section{Proofs for 
Sections \ref{s_cohbound} and \ref{s_levbound}}\label{s_app}
For the coherence-based bounds in Section~\ref{s_cohbound} we first
present a matrix concentration inequality (Section~\ref{s_chernoff}), 
and then the proofs of
Theorem~\ref{t_cohbound} (Section~\ref{s_tcohproof})
and Corollary~\ref{c_cohbound} (Section~\ref{s_ccohproof}).

For the bound based on leverage scores in Section~\ref{s_levbound},
we first present a matrix concentration inequality
(Section~\ref{s_bernstein}), and then the proof of
Theorem~\ref{t_levbound} (Section~\ref{s_tlevproof}).
\subsection{Matrix Chernoff Concentration inequality}\label{s_chernoff}
The matrix concentration inequality below is the basis for
Theorem~\ref{t_cohbound} and Corollary~\ref{c_cohbound}.

Denote the eigenvalues of a Hermitian matrix $Z$ by $\lambda_j(Z)$,
and the smallest and largest eigenvalues by
$\lambda_{min}(Z)\equiv \min_j{\lambda_j(Z)}$ and
$\lambda_{max}(Z) \equiv \max_j{\lambda_j(Z)}$, respectively.

\begin{theorem}[Corollary 5.2 in \cite{tropp11}]\label{t_tropp}
Let $X_j$ be a finite number of independent random $n\times n$
Hermitian positive semidefinite matrices with 
$\max_j{\|X_j\|_2}\leq \tau$.
Define 
$$\omega_{min}\equiv\lambda_{min}\left(\sum_j{\E[X_j]}\right) \qquad
\omega_{max}\equiv\lambda_{max}\left(\sum_j{\E[X_j]}\right),$$
and $f(x)\equiv e^x(1+x)^{-(1+x)}$.
Then for any $0\leq \epsilon< 1$
$$\Prob\left[\lambda_{min}\left(\sum_j{X_j}\right) 
\leq (1-\epsilon) \>\omega_{min} \right]\leq n 
\> f(-\epsilon)^{\omega_{min}/\tau},$$
and for any $\epsilon\geq 0$
$$\Prob\left[\lambda_{max}\left(\sum_j{X_j}\right) 
\geq (1+\epsilon) \>\omega_{max} \right]\leq n  \> 
f(\epsilon)^{\omega_{max}/\tau}.$$
\end{theorem}

\subsection{Proof of Theorem~\ref{t_cohbound}}\label{s_tcohproof}
We present a separate proof for each sampling method.

\paragraph{Algorithm~\ref{alg_without}: Sampling without replacement}
The proof follows directly from \cite[Lemma 3.4]{tropp11b}.

\paragraph{Algorithm~\ref{alg_with}: Sampling with replacement}
The proof is based on Theorem~\ref{t_tropp}, and
turns out to be somewhat similar to that of \cite[Lemma 3.4]{tropp11b}.

Set $X_t\equiv\frac{m}{c}\>Q^Te_{k_t}e_{k_t}^TQ$, $1\leq t\leq c$.
Then $X_t$ is $n\times n$  Hermitian positive semidefinite and
$\|X_t\|_2\leq\tfrac{m}{c}\|e_{k_t}^TQ\|_2^2\leq \tfrac{m\mu}{c}$.
Hence we set $\tau=m\mu/c$.
Furthermore,
$$\E[X_t]=\sum_{j=1}^m{\tfrac{1}{m}\>
\left(\tfrac{m}{c}\>Q^Te_je_j^TQ\right)}=
\tfrac{1}{c}\>\sum_{j=1}^m{Q^Te_je_j^TQ}=\tfrac{1}{c} I_n.$$
Hence the eigenvalues of the sum are
$\lambda_j\left(\sum_{t=1}^c{\E[X_t]}\right)=\lambda_j(I_n)=1$,
$1\leq j\leq n$, and we set $\omega_{min}=\omega_{max}=1$.
Applying Theorem~\ref{t_tropp} to $\sum_{t=1}^c{X_t}=Q^TS^TSQ$ gives
\begin{eqnarray*}
\Prob\left[\lambda_{min}\left(Q^TS^TSQ\right) 
\leq 1-\epsilon \right]\leq\ &  n f(-\epsilon)^{c/(m\mu)}\\
\Prob\left[\lambda_{max}\left(Q^TS^TSQ\right) 
\geq 1+\epsilon \right]\leq\ & n f(\epsilon)^{c/(m\mu)}.
\end{eqnarray*}
The result follows from Boole's inequality \cite[p. 16]{Ross}.

\paragraph{Algorithm~\ref{alg_bernoulli}: Bernoulli sampling}
The proof is similar to the one above, and a special case of
\cite[Theorem 6.1]{GT11}. 

Set
$$X_j\equiv\tfrac{m}{c}\> \begin{cases}
Q^Te_je_j^TQ & \text{with probability $\frac{c}{m}$}\\
0_{n\times n}& \text{with probability $1-\frac{c}{m}$} 
\end{cases}, \qquad 1\leq j\leq m.$$
Then $X_j$ is $n\times n$  Hermitian positive semidefinite,
$\|X_j\|_2\leq\tfrac{m}{c}\|e_j^TQ\|_2^2\leq \tfrac{m\mu}{c}$.
As above, we set $\tau=m\mu/c$.
Furthermore,
$$\E[X_j]=\tfrac{c}{m}\cdot \tfrac{m}{c}Q^Te_je_j^TQ
+(1-\tfrac{c}{m})\cdot 0_{n\times n}=Q^Te_je_j^TQ,$$
which implies $\sum_{j=1}^m{\E[X_j]}=\sum_{j=1}^m{Q^Te_je_j^TQ}=I_n$.
Now proceed as in the above proof for Algorithm~\ref{alg_with}, and 
apply Theorem~\ref{t_tropp} to $\sum_{j=1}^m{X_j}=Q^TS^TSQ$.

\subsection{Proof of Corollary~\ref{c_cohbound}}\label{s_ccohproof}
First we simplify the bound in Theorem~\ref{t_cohbound} based on the
inequality $f(-x)\leq f(x)$ for $0<x<1$ .
This implies for Theorem~\ref{t_cohbound} that
$$\delta\equiv n\left(f(-\epsilon)^{c/(m\mu)}+
f(\epsilon)^{c/(m\mu)}\right)
\delta\leq 2n\>f(\epsilon)^{c/(m\mu)}.$$
Solving for $c$ gives 
\begin{eqnarray*}
c\geq m\mu\>\frac{\ln(2n/\delta)}{-\ln{f(\epsilon)}}.
\end{eqnarray*}
If we can show that $-\ln{f(\epsilon)}>\epsilon^2/3$, then
the above lower bound for $c$ definitely holds if
$$c\geq 3m\mu\>\frac{\ln(2n/\delta)}{\epsilon^2}.$$

To show $-\ln{f(\epsilon)}>\epsilon^2/3$ for $0<\epsilon<1$, apply 
the definition $f(x)= e^x(1+x)^{-(1+x)}$ so that
$h(x)\equiv -\ln{f(x)} = (1+x)\ln{(1+x)}-x$. 
Expand into the power series 
$\ln{(1+x)}=\sum_{j=1}^{\infty}{(-1)^{j+1} \tfrac{x^j}{j}}$.
For $0<x< 1$ this yields
$h(x)=\frac{1}{2}x^2-\frac{1}{6}x^3+E(x)$, where 
$$E(x)\equiv\sum_{j=4}^{\infty}{(-1)^j\>\frac{x^j}{(j-1)j}}=
\sum_{j=2}^{\infty}{
\left(\frac{2j+1-(2j-1)x}{(2j-1)2j(2j+1)}\right)\> x^{2j}}>0,$$
since each summand is positive for $0<x< 1$.
Thus for $0<x< 1$ we obtain
$$h(x)>\frac{1}{2}x^2-\frac{1}{6}x^3=
\frac{3-x}{6}\>x^2\geq \frac{x^2}{3}.$$

\subsection{Matrix Bernstein concentration inequality}\label{s_bernstein}
The matrix concentration inequality below is the basis for
Theorem~\ref{t_levbound}.
It is a version specialized to square matrices of \cite[Theorem 4]{Recht11}.
In numerical experiments we found it to be tighter than 
\cite[Theorem 4]{DMMS10} 
and the Frobenius norm bound \cite[Theorem 2]{DKM06}.

\begin{theorem}[Theorem 4 in \cite{Recht11}]
Let $X_j$ be $m$ independent random $n\times n$ matrices
with $\E[X_j]=0_{n\times n}$, $1\leq j\leq m$. Let 
$\rho_j\equiv \max\{\|\E[X_jX_j^T]\|_2, \, \|\E[X_j^TX_j]\|_2\}$ and
$\max_{1\leq j\leq m}{\|X_j\|_2}\leq \tau$.
Then for any $\epsilon >0$ 
$$\Prob\left[\|\sum_{j=1}^m{X_j}\|_2> \epsilon\right]\leq
2n\>\exp\left(-\tfrac{3}{2}\>\frac{\epsilon^2}{3\sum_{j=1}^m{\rho_j}
+\tau\epsilon}\right).$$
\end{theorem}

\subsection{Proof of Theorem~\ref{t_levbound}}\label{s_tlevproof}
The proof is similar to that of \cite[Lemma 3]{BalzRN11}. 
Represent the outcome of uniform sampling with 
replacement in Algorithm~\ref{alg_with} by
$Q^TS^TSQ=\sum_{t=1}^c{Y_t}$, where
$Y_t\equiv \frac{m}{c}\>Q^Te_{k_t}e_{k_t}^TQ$ are $n\times n$ matrices, 
$1\leq t\leq c$, with expected value
$$\E[Y_t]=\sum_{j=1}^m{\frac{1}{m}\>\frac{m}{c} Q^Te_je_j^TQ}
=\frac{1}{c}\>\sum_{j=1}^m{Q^Te_je_j^TQ}=\frac{1}{c} I_n.$$
Thus, the zero mean versions are
$X_t\equiv Y_t-\tfrac{1}{c}I_n$.
To apply Theorem \cite[Theorem 4]{Recht11} to the $X_t$ we need to verify 
that they
fulfill the required conditions. 
First, by construction, $\E[X_t]=0$, $1\leq t\leq c$. Second,
since $Y_t$ and $I_n$ are symmetric positive semidefinite,
$$\|X_t\|_2\leq \max\{\|Y_t\|_2, \|\frac{1}{c}I_n\|_2\}=
\frac{1}{c}\max\{m\>\|e_{k_t}^TQ\|_2^2,1\}\leq \frac{m\mu}{c},$$
where the last inequality follows from the definition of $\mu$, and 
$\mu\geq n/m$. Hence we set $\tau=m\mu/c$.
Third, since $X_t$ is symmetric,
$$X_t^TX_t=X_tX_t^T=X_t^2=Y_t^2 -\frac{2}{c} Y_t+\frac{1}{c^2}I_n.$$
From $\E[Y_t]=\tfrac{1}{c} I_n$ follows
\begin{eqnarray}\label{e_ys}
\E[X_t^2]=\E[Y_t^2]-\frac{2}{c}\>\E[Y_t]+\frac{1}{c^2}I_n=
\E[Y_t^2]-\frac{1}{c^2}I_n.
\end{eqnarray}
Since $Y_t^2=\tfrac{m^2}{c^2}\>\ell_{k_t}\>Q^Te_{k_t}e_{k_t}^TQ$, we obtain
\begin{eqnarray*}
\E[Y_t^2]=\sum_{j=1}^m{\frac{1}{m}\> \frac{m^2}{c^2}\>\ell_j Q^Te_je_j^TQ}
=\frac{m}{c^2}\>Q^T \> \left(\sum_{j=1}^m{\ell_je_je_j^T}\right)\>Q
= \frac{m}{c^2}\> Q^TLQ.
\end{eqnarray*}
Substituting this into (\ref{e_ys}) yields
$$\E[X_t^2]=\tfrac{1}{c^2} \> \left(m\> Q^TLQ - I_n\right).$$
Positive semi-definiteness gives
$$\|\E[X_t^2]\|_2\leq \frac{1}{c^2}\>\max\{m\>\|Q^TLQ\|_2,\, 1\}=
\frac{m}{c^2}\>\|Q^TLQ\|_2.$$
We set $\rho_t=\tfrac{m}{c^2}\> \|Q^TLQ\|_2$.
Applying \cite[Theorem 4]{Recht11}  to
$$\sum_{t=1}^c{X_t}=\sum_{t=1}^c{\left(Y_t-\tfrac{1}{c}I_n\right)}=
(SQ)^T(SQ)- I_n$$
shows that
$\|\sum_{t=1}^c{X_t}\|_2\leq \epsilon$ with probability at least
$1-\delta$.

\section{Two-norm bound for scaled matrices, and proofs
for Sections \ref{s_clevbound} and~\ref{s_ancomp}}\label{s_app2}
We derive a bound for the two-norm of diagonally scaled matrices
(Section~\ref{s_tns}),
which leads immediately to the proofs
of Corollary~\ref{c_1} (Section~\ref{s_lqproof}), and
Corollary~\ref{c_comp} (Section~\ref{s_compproof}).

\subsection{Bound}\label{s_tns}
We present two majorization bounds for Hadamard products of vectors
(Lemmas \ref{l_s1} and~\ref{l_s2}), and use them to 
derive a bound for the two-norm of diagonally scaled matrices
(Theorem~\ref{t_2}). 

\begin{definition}[Definition 4.3.41 in \cite{HoJ12}]\label{d_maj}
Let $a$ and $b$ be vectors with $m$ real elements. The 
elements, labelled in algebraically decreasing order, are 
$a_{[1]}\geq \cdots\geq a_{[m]}$ and 
$b_{[1]}\geq \cdots\geq b_{[m]}$. 
The vector $a$ {\rm weakly majorizes} the vector $b$, if 
$$\sum_{j=1}^k{a_{[j]}}\geq \sum_{j=1}^k{b_{[j]}}, \qquad 1\leq k\leq m.$$
The vector $a$ {\rm majorizes} the vector $b$, if 
$a$ weakly majorizes $b$ and also $\sum_{j=1}^m{a_{[j]}}=\sum_{j=1}^m{b_{[j]}}$.
\end{definition}

The first lemma follows from a stronger majorization inequality for 
functions that are monotone and lattice superadditive.

\begin{lemma}[Theorem II.4.2 in \cite{Bha97}]\label{l_s2}
If $b$ and $x$ are vectors with $m$ non-negative elements, then
$$\sum_{j=1}^k{b_j\, x_j}\leq \sum_{j=1}^k{b_{[j]}\, x_{[j]}},
\qquad 1\leq k\leq m.$$
\end{lemma}

The second lemma is a variant of a well-known majorization result for 
Hadamard products of vectors \cite[Lemma 4.3.51]{HoJ12}. Since the
version below is slightly different, we include a proof from first principles.

\begin{lemma}\label{l_s1}
Let $x$, $a$ and $b$ be vectors with $m$ non-negative elements.
If $a$ weakly majorizes $b$, then
$$\sum_{j=1}^k{a_{[j]}\, x_{[j]}}\geq \sum_{j=1}^k{b_{[j]}\, x_{[j]}},
\qquad 1\leq k\leq m.$$
\end{lemma}

\begin{proof}
The following arguments hold for $1\leq k\leq m-1$. 
Start out with the upper bound, and separate the last summand,
\begin{eqnarray}\label{e_m3}
\sum_{j=1}^{k+1}{a_{[j]}\,x_{[j]}} &=&
\sum_{j=1}^{k}{a_{[j]}\,x_{[j]}}  + a_{[k+1]}\,x_{[k+1]}.
\end{eqnarray}
Re-writing the right sum and applying 
$x_{[j]}\geq x_{[k+1]}\geq 0$, $1\leq j\leq k$, gives
\begin{eqnarray*}
\sum_{j=1}^{k}{a_{[j]}\,x_{[j]}}  &=& \sum_{j=1}^{k}{b_{[j]}\,x_{[j]}}  +
\sum_{j=1}^{k}{(a_{[j]}-b_{[j]})\,x_{[j]}}\\
&\geq & \sum_{j=1}^{k}{b_{[j]}\,x_{[j]}}  +
\sum_{j=1}^{k}{(a_{[j]}-b_{[j]})}\, x_{[k+1]} \\
&= & \sum_{j=1}^{k}{b_{[j]}\,x_{[j]}}  +
\left(\sum_{j=1}^{k}{a_{[j]}}-\sum_{j=1}^{k}{b_{[j]}}\right)\, x_{[k+1]}.\\
\end{eqnarray*}
Insert this into (\ref{e_m3}) and gather common terms,
\begin{eqnarray*}
\sum_{j=1}^{k+1}{a_{[j]}\,x_{[j]}} &\geq &
\sum_{j=1}^{k}{b_{[j]}\,x_{[j]}}  +
\left(\sum_{j=1}^{k+1}{a_{[j]}}  
-\sum_{j=1}^{k}{b_{[j]}}\right)\, x_{[k+1]}\\
&\geq &\sum_{j=1}^{k}{b_{[j]}\,x_{[j]}}  + b_{[k+1]}\,x_{[k+1]}
=\sum_{j=1}^{k+1}{b_{[j]}\,x_{[j]}},
\end{eqnarray*}
where the second inequality follows from the majorization
$\sum_{j=1}^{k+1}{a_{[j]}}\geq \sum_{j=1}^{k+1} {b_{[j]}}$.
\end{proof}

Now we are ready to bound the two norm of a 
row scaled matrix $DZ$, where  $Z$ is $m\times n$ of full column rank, and
$D=\diag\begin{pmatrix}d_1 &\ldots & d_m\end{pmatrix}$
is a non-negative  $m\times m$  diagonal matrix. The obvious bound is
\begin{eqnarray}\label{e_sub}
\|DZ\|_2\leq \|D\|_2\>\|Z\|_2=d_{[1]}\>\|Z\|_2.
\end{eqnarray}
However, the bound in Theorem~\ref{t_2} below, which 
incorporates the largest row norm of $Z$ and several of the largest 
(in magnitude) diagonal elements  of~$D$, turns out to be tighter.

\begin{theorem}\label{t_2}
Let $Z$ be a real $m\times n$ matrix with $\rank(Z)=n$,
smallest singular value $\sigma_z=1/\|Z^{\dagger}\|_2$, and largest
squared row norm $\mu_z\equiv\max_{1\leq j\leq m}{\|e_j^TZ\|_2^2}$. If
$t\equiv\left\lfloor\sigma_z^2/\mu_z\right\rfloor$, then
$$\|DZ\|_2^2\leq 
\begin{cases}
\mu_z\sum_{j=1}^t{d_{[j]}^2}+\left(\|Z\|_2^2-t\,\mu_z\right)\,d_{[t+1]}^2 & 
\text{if}~ \|Z\|_2^2-t\,\mu_z\leq \mu_z \\
\mu_z\sum_{j=2}^{t+1}{d_{[j]}^2}+\left(\|Z\|_2^2-t\,\mu_z\right)\,d_{[1]}^2 &
\text{otherwise}.\end{cases}$$
\end{theorem}

\begin{proof}
Let $z$ be a $n\times 1$ vector with $\|z\|_2=1$ and
$\|DZ\|_2=\|DZz\|_2$. Furthermore let $z_j\equiv e_j^TZz$, $1\leq j\leq m$, be
the elements of $Zz$, so that $\|Zz\|_2^2=\sum_{j=1}^m{z_j^2}$.
\smallskip

\paragraph{Apply Lemma~\ref{l_s2}}
Since $d_j^2\geq 0$ and $z_{j}^2\geq 0$, $1\leq j\leq m$, we can
apply Lemma~\ref{l_s2} with $x_j=d_{j}^2$ and $b_j=z_{j}^2$, to obtain
\begin{eqnarray*}
\|DZ\|_2^2=\|DZz\|_2^2= \sum_{j=1}^m{d_{j}^2\,z_{j}^2}=
\sum_{j=1}^m{b_j\, x_j}\leq \sum_{j=1}^m{b_{[j]}\,x_{[j]}}.
\end{eqnarray*}

\paragraph{Verify assumptions of Lemma~\ref{l_s1}}
In order to apply Lemma~\ref{l_s1} with 
$$a_j=\mu_z,  \quad 1\leq j\leq t,\qquad  a_{t+1}=\|Zz\|_2^2-t\,\mu_z,
\qquad a_j=0, \quad t+2\leq j\leq m,$$
we need show that the assumptions are satisfied,
meaning all  vector elements are non-negative and the
majorization condition holds. Clearly
$a_j\geq 0$ for  $1\leq j\leq t$ and $t+2\leq j\leq m$.
This leaves $a_{t+1}$. From $\rank(Z)=n$ follows that $\sigma_z>0$.
The definition of $t$ implies $0\leq t\leq \sigma_z^2/\mu_z$, so that
$$0\leq \sigma_z^2-t\,\mu_z=\min_{\|y\|_2=1}{\|Zy\|_2^2}-t\,\mu_z\leq 
\|Zz\|_2^2-t\,\mu_z=a_{t+1}.$$
Thus, all vector elements are non-negative.

To show the majorization condition, start with the 
Cauchy-Schwartz inequality,
$$b_j=z_{j}^2=(e_{j}^TZ\>z)^2\leq  
\|e_{j}^TZ\|_2^2\>\|z\|_2^2=\|e_{j}^TZ\|_2^2\leq \mu_z, \qquad
1\leq j\leq m.$$
This yields, regardless of whether
$a_{t+1}=\|Zz\|_2^2-t\,\mu_z\leq \mu_z$ or not, 
$$\sum_{j=1}^k{a_{[j]}}\geq\sum_{j=1}^k{\mu_z}\geq \sum_{j=1}^k{z_{[j]}^2}=
\sum_{j=1}^k{b_{[j]}},\qquad 1\leq k\leq t.$$

Moreover, for $1\leq k\leq m-t$,
$$\sum_{j=1}^{t+k}{a_{[j]}}=\sum_{j=1}^t{\mu_z}+(\|Zz\|_2^2-t\,\mu_z)
=\|Zz\|_2^2\geq \sum_{j=1}^{t+k}{z_{[j]}^2}=\sum_{j=1}^{t+k}{b_{[j]}}.$$
This gives the weak majorization condition
$\sum_{j=1}^k{a_{[j]}}\geq \sum_{j=1}^k{b_{[j]}}$, $1\leq k\leq m$.
\smallskip

\paragraph{Apply Lemma \ref{l_s1}}
Since the assumptions of Lemma~\ref{l_s1} are satisfied, we can conclude
that
$\sum_{j=1}^{m}{b_{[j]}\,x_{[j]}}\leq \sum_{j=1}^{m}{a_{[j]}\,x_{[j]}}$.
At last, substitute into this majorization 
relation the expressions for $a$ and $b$.
If $\|Z\|_2^2-t\,\mu_z \leq \mu_z$, then
$$\sum_{j=1}^m{a_{[j]}\,x_{[j]}}=\mu_z\sum_{j=1}^t{d_{[j]}^2}+
(\|Zz\|_2^2-t\,\mu_z)\,d_{[t+1]}^2\leq 
\mu_z\sum_{j=1}^t{d_{[j]}^2}+(\|Z\|_2^2-t\,\mu_z)\,d_{[t+1]}^2,$$
otherwise
$$\sum_{j=1}^m{a_{[j]}\,x_{[j]}}=(\|Zz\|_2^2-t\,\mu_z)\,d_{[1]}^2+
\mu_z\sum_{j=2}^{t+1}{d_{[j]}^2}\leq 
(\|Z\|_2^2-t\,\mu_z)\,d_{[1]}^2+
\mu_z\sum_{j=2}^{t+1}{d_{[j]}^2}.$$
\end{proof}

Theorem~\ref{t_2} is tighter than (\ref{e_sub}) because
$d_{[j]}^2\leq \|D\|_2^2$ implies
\begin{eqnarray*}
\|DZ\|_2^2 & \leq &
\begin{cases}
\mu_z\sum_{j=1}^t{d_{[j]}^2}+\left(\|Z\|_2^2-t\,\mu_z\right)\,d_{[t+1]}^2 & 
\text{if}~ \|Z\|_2^2-t\,\mu_z\leq \mu_z \\
\mu_z\sum_{j=2}^{t+1}{d_{[j]}^2}+\left(\|Z\|_2^2-t\,\mu_z\right)\,d_{[1]}^2 &
\text{otherwise}\end{cases}\\
& \leq & t\mu_z \|D\|_2^2 -\left(\|Z\|_2^2-t\,\mu_z\right)\,\|D\|_2^2 =
\|D\|_2^2\|Z\|_2^2.
\end{eqnarray*}

\begin{corollary}\label{c_t2}
Let $Z$ be a real $m\times n$ matrix with $Z^TZ=I_n$, and coherence
$\mu_z\equiv\max_{1\leq j\leq m}{\|e_j^TZ\|_2^2}$. If
$t\equiv\left\lfloor 1/\mu_z\right\rfloor$, then
$$\|DZ\|_2^2\leq \mu_z\sum_{j=1}^t{d_{[j]}^2}+
\left(1-t\,\mu_z\right)\,d_{[t+1]}^2.$$
\end{corollary}
\begin{proof}
Applying Theorem \ref{t_2} and assuming $1-t\,\mu_z\leq \mu_z$ gives
$$\|DZ\|_2^2\leq \mu_z\sum_{j=1}^t{d_{[j]}^2}+
\left(1-t\,\mu_z\right)\,d_{[t+1]}^2.$$
The assumption $1-t\,\mu_z\leq \mu_z$ is justified because
$$1-t\,\mu_z = 1-\left\lfloor 1/\mu_z\right\rfloor\,\mu_z \leq 1-(1/\mu_z - 1)\mu_z = \mu_z.$$
\end{proof}

\subsection{Proof of Corollary~\ref{c_1}}\label{s_lqproof}
Apply Corollary~\ref{c_t2} with $D=L^{1/2}$, $Z=Q$, $\mu_z=\mu$,
and $t=\lfloor 1/\mu\rfloor$ to prove the first inequality,
$$\|Q^TLQ\|_2=\|L^{1/2}Q\|_2^2\leq \mu \>
\sum_{j=1}^t{\ell_{[j]}}+(1-t\,\mu)\,\ell_{[t+1]}.$$ 
As for the second inequality, $\ell_{[j]}\leq \mu$ implies
$$\mu\>\sum_{j=1}^t{\ell_{[j]}}+(1-t\,\mu)\,\ell_{[t+1]}\
\leq t\mu^2+ (1-t\mu)\mu=\mu.$$
If, in addition, $t$ is an integer, then $t=1/\mu$ and $1-t\,\mu=0$.

\subsection{Proof of Corollary~\ref{c_comp}}\label{s_compproof}
Define the common term 
$\phi\equiv m\>\ln(2n/\delta)/\epsilon^2$
in both bounds, and write Corollary~\ref{c_cohbound} as
$c\geq  3\mu\>\phi$,
and Corollary~\ref{c_levbound} as
$c\geq  (2\tau+\tfrac{2}{3}\epsilon\,\mu)\>\phi$.
From $\epsilon<1$ and $\tau\leq \mu$ follows
$$2\tau+\tfrac{2}{3}\epsilon\,\mu\leq 3\mu.$$

\section{Existence of matrices with prescribed coherence and 
leverage scores}\label{s_pre}
This section is the basis for Algorithm~\ref{alg_gen}.
We review a well-known majorization result (Theorem~\ref{t_horn}). 
We use it to show (Theorem~\ref{t_maj}) that, given prescribed matrix 
dimensions and leverage scores, there always exists a
matrix~$Q$ with orthonormal columns that has the required dimensions
and (squared) row norms equal to the leverage scores.

Our approach is again based on majorization, see Definition~\ref{d_maj},
and in particular on the fact that the eigenvalues of
a real symmetric matrix majorize its diagonal elements.

\begin{theorem}[Theorem 4.3.48 in \cite{HoJ12}]\label{t_horn}
Let $a$ and $\lambda$ be  vectors with real elements 
$a_j$ and $\lambda_j$, respectively, $1\leq j\leq m$.
If $\lambda$ majorizes $a$, then there exists 
a $m\times m$ real symmetric matrix with 
eigenvalues $\lambda_j$ and diagonal elements $a_j$, $1\leq j\leq m$.
\end{theorem}

With the help of Theorem~\ref{t_horn} we show that there exists
a matrix with orthonormal columns that has prescribed leverage scores and
coherence.

\begin{theorem}\label{t_maj}
Given integers $m$ and $n$ with $m\geq n\geq 1$; 
and a vector $\ell$ with $m$ elements $\ell_j$ that satisfy
$0\leq \ell_j \leq 1$ and $\sum_{j=1}^m{\ell_j}=n$. Then
there exists a $m\times n$ matrix $Q$ with orthonormal 
columns that has leverage scores $\|e_j^TQ\|_2^2=\ell_j$, $1\leq j\leq m$,
and coherence $\mu=\max_{1\leq j\leq m}{\ell_j}$.
\end{theorem}

\begin{proof}
Let $\lambda$ be a vector with $m$ elements that satisfy
$\lambda_j=1$ for $1\leq j\leq n$, and $\lambda_j=0$ for
$n+1\leq j\leq m$.
We are going to construct a matrix $Q$ by applying Theorem~\ref{t_horn} 
to $\lambda$ and $\ell$.
To this end, we first need to show that $\lambda$ majorizes $\ell$.

\paragraph{Majorization}
We distinguish the cases  $1\leq k\leq n$ and $n+1\leq k\leq m$.
\begin{description}
\item[Case $1\leq k\leq n$:\ ] From $\ell_j\leq 1$ follows
$$\sum_{j=1}^{k}{\lambda_j}=k\geq \sum_{j=1}^k{\ell_{[j]}}.$$
\item[Case $n+1\leq k\leq m$:\ ]
From $\ell_j\geq 0$ and $\sum_{j=1}^m{\ell_j}=n$ follows
$$\sum_{j=1}^{k}{\lambda_j}=n=\sum_{j=1}^{k}{\ell_{[j]}}  +  
\sum_{j=k+1}^{m}{\ell_{[j]}} \geq \sum_{j=1}^{k}{\ell_{[j]}}.$$
\end{description}
Hence 
$$\sum_{j=1}^{k}{\lambda_j}\geq \sum_{j=1}^k{\ell_{[j]}}, 
\qquad 1\leq k\leq m,$$
which means that $\lambda$ weakly majorizes $\ell$.
Since also $\sum_{j=1}^m{\lambda_j}=n=\sum_{j=1}^{m}{\ell_{[j]}}$,
we can conclude that $\lambda$ majorizes $\ell$.

\paragraph{Construction of $Q$}
Theorem \ref{t_horn} implies that there 
exists a real symmetric matrix $W$ with  eigenvalues
$\lambda_j$ and diagonal elements $W_{jj}=\ell_j$, $1\leq j\leq m$.
Since $W$ has $n$ eigenvalues equal to one, and all other 
eigenvalues equal to zero, it has an eigenvalue decomposition
$$W=\hat{Q}\begin{pmatrix}I_n &0\\ 0&0\end{pmatrix}\hat{Q}^T=QQ^T,$$
where $\hat{Q}$ is a $m\times m$ real orthogonal matrix, and 
$Q\equiv\hat{Q}\begin{pmatrix}I_n&0\end{pmatrix}^T$ has $n$ orthonormal
columns. Therefore $Q$ has leverage scores 
$\|e_j^TQ\|_2^2=e_j^TQQ^Te_j=W_{jj}=\ell_j$
and coherence $\mu=\max_{1\leq j\leq m}{\ell_j}$.
\end{proof}
\bibliography{biblio} 
\end{document}